\documentclass[12pt]{article}\usepackage{amssymb}
\usepackage{amsthm,amsmath}

\newtheorem{theorem}{Theorem}[section]

\newtheorem{corollary}[theorem]{Corollary}

\newtheorem{lemma}[theorem]{Lemma}

\newtheorem{proposition}[theorem]{Proposition}

\theoremstyle{definition}
\newtheorem{definition}[theorem]{Definition}
\newtheorem{example}[theorem]{Example}
\newtheorem{remark}[theorem]{Remark}
\newcommand{\F}{\mathbb{F}}

\newcommand{\Char}{\mathop{\rm char}\nolimits}

\begin{document}

\title{Linear degenerations of algebras and certain representations of the
general linear group}
\author{Christakis A. Pallikaros$^*$ and Harold N. Ward$^{\dag}$}
\maketitle

\begin{abstract}
Let $\boldsymbol{\Lambda}\,(=\mathbb{F}^{n^{3}})$, where $\mathbb{F}$ is a field with
$|\mathbb{F}|>2$, be the space of structure vectors of algebras having the $n$-dimensional $\mathbb{F}$-space $V$ as the underlying vector space. 
Also let
$G=GL(V)$. Regarding $\boldsymbol{\Lambda}$ as a $G$-module via the `change
of basis' action of~$G$ on~$V$, we determine the composition factors of various $G$-submodules 
of~$\boldsymbol{\Lambda}$ which correspond to certain important families of
algebras.
This is achieved by introducing the notion of linear degeneration which
allows us to obtain analogues over $\F$ of certain known results on degenerations of algebras.
As a result, the $GL(V)$-structure of~$\boldsymbol{\Lambda}$ is determined.
\end{abstract}


{\vspace{1mm}}

\noindent {\footnotesize \textit{{\ ${}^{\ast }$~Department of Mathematics
and Statistics, University of Cyprus, PO Box 20537, 1678 Nicosia, Cyprus%
\newline
\phantom{${}^{*}$~}{}E-mail: pallikar@ucy.ac.cy\newline
${}^{\dag }$~Department of Mathematics, University of Virginia,
Charlottesville, VA 22904, USA\newline
\phantom{${}^{\dag}$~}{}E-mail: hnw@virginia.edu\newline
}}}



\section{Introduction}

The concept of degeneration has important links with various branches of
mathematics, representation theory being one of them (see, for example,~\cite%
{Kraft1982}). In~\cite{Gorbatsevich1991} Gorbatsevich classified, up to
isomorphism, all $n$-dimensional skew-symmetric algebras over $\mathbb{C}$
which have the Abelian Lie algebra as their only proper degeneration. In
order to achieve this, he made use the theory of algebraic groups and their
representations (see, for example,~\cite{Geck2003}, \cite{Humphreys1991}),
which allowed him to locate various such `level 1' algebras.

In this paper, in some sense, we take a `reverse' direction to that taken in~%
\cite{Gorbatsevich1991}. Our aim is to obtain information about certain
representations of the general linear group, where the representations and
the group are defined over an arbitrary field $\mathbb{F}$, having as a
starting point certain known results on degenerations.
Our motivation comes from~\cite{IvanovaPallikaros2019}, in particular the
way certain results on degenerations of algebras over an arbitrary infinite
field obtained in that paper were used in order to extract information on
the composition series of a certain representation of the general linear
group defined over an arbitrary infinite field. This representation
naturally corresponds to the class of `skew' algebras (see~\cite[Section~4.1]%
{IvanovaPallikaros2019}).

It will be convenient at this point to introduce some notation and recall
some terminology. Let $V$ be an $n$-dimensional vector space over an
arbitrary field $\mathbb{F}$ and let $G=GL(V)$. As in~\cite{Gorbatsevich1991}
and~\cite{IvanovaPallikaros2019} we will be considering the natural `change
of basis' action of $G$ on $\boldsymbol{\Lambda}\,(=\mathbb{F}^{n^3})$, the
space of structure vectors of algebras having $V$ as the underlying space.
This is a linear action of $G$ on $\boldsymbol{\Lambda}$. Recall that for $%
\boldsymbol{\lambda}_1, \boldsymbol{\lambda}_2\in\boldsymbol{\Lambda}$, we
say that $\boldsymbol{\lambda}_1$ degenerates to $\boldsymbol{\lambda}_2$ if
$\boldsymbol{\lambda}_2$ belongs to the Zariski-closure of the $G$-orbit of $%
\boldsymbol{\lambda}_1$ (relative to the above action of $G$). The notion of
degeneration is useful only in the case the field $\mathbb{F}$ is infinite
since everything is closed when $\mathbb{F}$ is finite. As the techniques
used in~\cite[Section~4.1]{IvanovaPallikaros2019} rely heavily on
degenerations, the standing assumption there is that $\mathbb{F}$ is an
(arbitrary) infinite field. One of the main contributions of the present
paper is that, by using an approach which is uniform for $\mathbb{F}$ finite
and $\mathbb{F}$ infinite, the $G$-submodule structure of various submodules
of $\boldsymbol{\Lambda}$ corresponding to certain important classes of
algebras is completely determined (and hence the $G$-module structure of $%
\boldsymbol{\Lambda}$ itself). This is achieved by introducing the notion of
`linear degeneration' which allows us to obtain `linear degeneration
analogues', now over an arbitrary field $\mathbb{F}$ with $|\mathbb{F}|>2$,
of certain results in~\cite{IvanovaPallikaros2019} on degenerations.
Moreover, the use of tools like the adjoint trace form turns out to play a
key role as it allows us to obtain more detailed information (compared to
just using degenerations as in~\cite{IvanovaPallikaros2019}) on various
composition series even in the case $\mathbb{F}$ is infinite.

The paper is organized as follows: In Section~\ref{SectAlgSetUp} we develop
the general set-up for algebras and their ingredients and introduce some
notation. In Sections~\ref{SecSubmodulesCK}, \ref{SectAdjTrace}, \ref{SecM*}%
, \ref{SecM**} we introduce various $G$-submodules of $\boldsymbol{\Lambda}$
which correspond to some important classes of algebras and give defining
conditions and bases for them. Moreover, we show how the adjoint trace form
can be used to obtain information concerning various submodules of $%
\boldsymbol{\Lambda}$ via certain $G$-homomorphisms it allows us to define.
In Section~\ref{SectLinDegen} we introduce the notion of linear degeneration
and show how this can be used to obtain results, which are in a sense
`analogous' to certain results in~\cite{IvanovaPallikaros2019}, but which
are valid for any field with the only exception of some very small fields.
Using an action on a space of semilinear maps in Section~\ref{SectChar2} and
with the help of transvections in Section~\ref{SectTransvec}, we are able to
complete the proof of the various `linear degeneration analogues' we need,
for any field $\mathbb{F}$ with $|\mathbb{F}|>2$. Finally, in Section~\ref%
{SectGLVstructure}, we use the information obtained in the previous sections
in order to determine the $G$-structure of $\boldsymbol{\Lambda}$, the
approach being uniform for any field $\mathbb{F}$ with $|\mathbb{F}|>2$. In
order to achieve this, on the way, we obtain information about the
composition series of some of the important $G$-submodules of $\boldsymbol{%
\Lambda}$. In particular, we determine all composition series for the $G$%
-submodules corresponding to `commutative' and `skew' algebras.

\section{Algebra set-up}

\label{SectAlgSetUp}

In this section we introduce the general set-up for algebras and their
ingredients. The algebras are constructed on a vector space $V$ of dimension
$n$ over the field $\mathbb{F}$. The general linear group $GL(V)=G$ acts on
the left on $V$. We fix a basis $v_{1},\ldots ,v_{n}$ of $V$, which we will
refer to as the standard basis of $V$, and define its dual basis $\widehat{v}%
_{1},\ldots ,\widehat{v}_{n}$ in the usual way: $\widehat{v}%
_{i}(v_{j})=\delta _{ij}$. For $g\in G$, $gv_{j}=\sum_{i}g_{ij}v_{i}$.
Relative to the standard basis of $V$, the matrix for $g\in G$ is $\left[ g%
\right] =\left[ g_{ij}\right] $, and the coordinate vector of $%
v=\sum_i\xi_iv_i\in V$ is $\left[ v\right] =(\xi _{1},\ldots ,\xi _{n})^{T}$%
, a column vector ($^{T}$ for transpose). Thus $\left[ gv\right] =\left[ g%
\right] \left[ v\right] $.

\medskip

The action of $G$ on the dual space $\widehat{V}$ is on the right: for $%
\varphi \in \widehat{V},v\in V$, and $g\in G$, $(\varphi g)(v)=\varphi (gv)$%
. We thus have
\begin{equation*}
\widehat{v}_{i}(gv_{j})=\widehat{v}_{i}(\sum_{k}g_{kj}v_{k})=\sum_{k}g_{kj}%
\widehat{v}_{i}(v_{k})=\sum_{k}g_{kj}\delta _{ik}=g_{ij}.
\end{equation*}%
Hence $\widehat{v}_{i}g=\sum_{j}g_{ij}\widehat{v}_{j}$. So in matrix terms,
with respect to the dual basis $\widehat{v}_{1},\ldots ,\widehat{v}_{n}$
(identifying $\widehat{V}$ with $\mathbb{F}^{n}$ as a space of row-vectors),
the matrix for $g$ is still $[g]$, but multiplying on the right. Note that $%
\widehat{V}$ is irreducible as a right $G$-module since $G$ acts
transitively on $\widehat{V}-\{0\}$.

\medskip

A (not necessarily associative) algebra $\mathfrak{g}$ on $V$ has a bilinear
product $[\,,\,]$. The set of algebras $\mathbf{A}$ having $V$ as the
underlying vector space, forms itself a vector space over $\mathbb{F}$ by
the rules that the product for $\alpha \mathfrak{g}$ is $\alpha \left[ u,v%
\right] $, and the product for the sum $\mathfrak{g}_{1}+\mathfrak{g}_{2}$
is the sum of the products: $\left[ u,v\right] =\left[ u,v\right] _{1}+\left[
u,v\right] _{2}$. (If an algebra has a tag, we use the same tag on the
product symbol for the algebra. This also holds for the structure vectors
below.)

\begin{definition}
\label{DefAction} We define an action of $G$ on $\mathbf{A}$ by the rule
that for $\mathfrak{g}^{\prime }=\mathfrak{g}g$, the product is given by $%
\left[ u,v\right] ^{\prime }=g^{-1}[gu,gv]$. Writing this as $g\left[ u,v%
\right] ^{\prime }=[gu,gv]$, we see that $u\mapsto gu$ is an isomorphism
from $\mathfrak{g}^{\prime }$ to $\mathfrak{g}$.

The structure vector $\Theta (\mathfrak{g})$ of algebra $\mathfrak{g}$ in $%
\mathbf{A}$ is the member $\boldsymbol{\lambda }=(\lambda _{ijk})$ of $%
\mathbf{\Lambda }=\mathbb{F}^{n^{3}}$ with the components $\lambda _{ijk}$
being determined by the basis products: $\left[ v_{i},v_{j}\right]
=\sum_{k}\lambda _{ijk}v_{k}$. We define the action of $G$ on these vectors
by $\Theta (\mathfrak{g})g=\Theta (\mathfrak{g}g)$.
\end{definition}

It is easy to observe that the above actions of $G$ on $\mathbf{A}$ and $%
\boldsymbol{\Lambda}$ respectively are linear. In particular, the map $%
\Theta $ is a $G$-isomorphism from the right $G$-module $\mathbf{A}$ to the
right $G $-module $\boldsymbol{\Lambda}$.

\medskip

It is important to have a formula for $\Theta (\mathfrak{g})g$ in terms of $%
\Theta (\mathfrak{g})$ and the matrix $[g]$. If $\mathfrak{g}^{\prime }=%
\mathfrak{g}g$, then, assuming again that $\Theta (\mathfrak{g})=\boldsymbol{%
\lambda }=(\lambda _{ijk})$, we have
\begin{eqnarray*}
\left[ v_{i},v_{j}\right] ^{\prime } &=&g^{-1}[gv_{i},gv_{j}] \\
&=&g^{-1}\left[ \sum_{a}g_{ai}v_{a},\sum_{b}g_{bj}v_{b}\right] \\
&=&g^{-1}\sum_{a,b,c}g_{ai}g_{bj}\lambda _{abc}v_{c} \\
&=&\sum_{a,b,c}g_{ai}g_{bj}\lambda _{abc}g^{-1}v_{c}.
\end{eqnarray*}%
Denote $\Theta(\mathfrak{g}^{\prime})$ by $\boldsymbol{\lambda}%
^{\prime}=(\lambda_{ijk}^{\prime})$. Then $\boldsymbol{\lambda}%
^{\prime}=\Theta(\mathfrak{g}g)=\Theta (\mathfrak{g})g=\boldsymbol{\lambda }%
g $. Put $\left[ g^{-1}\right] =[g_{ij}^{(-1)}]$, write out $g^{-1}v_{c}$,
and expand the left with the structure coefficients for $\mathfrak{g}%
^{\prime }$ to get%
\begin{eqnarray*}
\sum_{k}\lambda _{ijk}^{\prime }v_{k} &=&\sum_{a,b,c}g_{ai}g_{bj}\lambda
_{abc}\sum_{k}g_{kc}^{(-1)}v_{k} \\
&=&\sum_{k}\left( \sum_{a,b,c}g_{ai}g_{bj}g_{kc}^{(-1)}\lambda _{abc}\right)
v_{k}.
\end{eqnarray*}%
That gives our formula:%
\begin{equation}  \label{basic}
\lambda _{ijk}^{\prime }=\sum_{a,b,c}g_{ai}g_{bj}g_{kc}^{(-1)}\lambda _{abc}.
\end{equation}
This formula can also be interpreted as giving the structure coefficients
for $\mathfrak{g}$ relative to the new basis $v_{1}^{\prime },\ldots
,v_{n}^{\prime }$ with $v_{j}^{\prime }=gv_{j}$. (Compare with~\cite[%
Definition~2.5 and Remark~2.6]{IvanovaPallikaros2019} but be aware of the
slight difference in notation, in particular regarding the standard basis of
$V$.)

\medskip

There is another way to picture things. The product in the algebra $%
\mathfrak{g}$ is a bilinear mapping from $V\times V$ to $V$. Such a mapping
corresponds to a member of $\widehat{V}\otimes \widehat{V}\otimes V$ by the
formula $(\varphi \otimes \psi \otimes w)(u,v)=\varphi (u)\psi (v)w$. If $%
\Theta (\mathfrak{g})=\boldsymbol{\lambda }\,(=(\lambda _{ijk}))$, we
consider the map%
\begin{equation*}
\chi\colon \boldsymbol{\lambda }\mapsto \sum_{i,j,k}\lambda _{ijk}(\widehat{v%
}_{i}\otimes \widehat{v}_{j}\otimes v_{k}),
\end{equation*}%
which correctly gives%
\begin{eqnarray*}
\left( \sum_{i,j,k}\lambda _{ijk}(\widehat{v}_{i}\otimes \widehat{v}%
_{j}\otimes v_{k})\right) (v_{x},v_{y}) &=&\sum_{i,j,k}\lambda _{ijk}\delta
_{ix}\delta _{jy}v_{k} \\
&=&\sum_{k}\lambda _{xyk}v_{k} \\
&=&[v_{x},v_{y}].
\end{eqnarray*}%
What about the $G$-action? It is on the right for the two $\widehat{V}$
factors, but it needs to be put on the right for $V$, and that is done by $%
vg:=g^{-1}v.$ With $\left[ g\right] =\left[ g_{xy}\right] $, we had $%
\widehat{v}_{x}g=\sum_{y}g_{xy}\widehat{v}_{y}$; and now $%
v_{y}g=g^{-1}v_{y}=\sum_{x}g_{xy}^{(-1)}v_{x}$. So%
\begin{eqnarray*}
(\widehat{v}_{a}\otimes \widehat{v}_{b}\otimes v_{c})g &=&\widehat{v}%
_{a}g\otimes \widehat{v}_{b}g\otimes g^{-1}v_{c} \\
&=&\sum_{i}g_{ai}\widehat{v}_{i}\otimes \sum_{j}g_{bj}\widehat{v}_{j}\otimes
\sum_{k}g_{kc}^{(-1)}v_{k} \\
&=&\sum_{i,j,k}g_{ai}g_{bj}g_{kc}^{(-1)}(\widehat{v}_{i}\otimes \widehat{v}%
_{j}\otimes v_{k}).
\end{eqnarray*}%
Thus%
\begin{eqnarray*}
\chi(\boldsymbol{\lambda })g &=&\sum_{a,b,c}\lambda _{abc}(\widehat{v}%
_{a}\otimes \widehat{v}_{b}\otimes v_{c})g \\
&=&\sum_{a,b,c}\lambda _{abc}\sum_{i,j,k}g_{ai}g_{bj}g_{kc}^{(-1)}(\widehat{v%
}_{i}\otimes \widehat{v}_{j}\otimes v_{k}) \\
&=&\sum_{i,j,k}\left( \sum_{a,b,c}\lambda
_{abc}g_{ai}g_{bj}g_{kc}^{(-1)}\right) (\widehat{v}_{i}\otimes \widehat{v}%
_{j}\otimes v_{k}) \\
&=&\sum_{i,j,k}\lambda _{ijk}^{\prime }(\widehat{v}_{i}\otimes \widehat{v}%
_{j}\otimes v_{k}) \\
&=&\chi(\boldsymbol{\lambda }^{\prime })=\chi(\boldsymbol{\lambda }g),
\end{eqnarray*}%
as it should be.

\medskip

\textbf{Notation.} Throughout the paper, we will assume that $n$ is a fixed
positive integer with $n\ge3$, and that $\mathbb{F}$ is an arbitrary field.
(For some of the results we will need to impose the restriction $|\mathbb{F}%
|>2$.) Unless otherwise stated, the $(i,j,k)$-component of the structure
vector $\boldsymbol{\lambda}$ will be denoted by $\lambda _{ijk}$. We will
use $\mathbf{abc}$ to mean the member $\boldsymbol{\lambda}%
\,(=(\lambda_{ijk}))$ of $\mathbf{\Lambda }$ having $\lambda _{abc}=1$ and
all other $\lambda _{ijk}$ equal to $0$. We will refer to the basis of $%
\boldsymbol{\Lambda}$ consisting of the $n^3$ structure vectors of this form
as the standard basis of $\boldsymbol{\Lambda}$.

\medskip

It will be convenient in various parts of the paper, in particular when we
give defining conditions or a basis for a $G$-submodule of $\boldsymbol{%
\Lambda}$, to use the following:

\medskip

\textbf{Convention ($\ddag$).} Different letters in the subscripts for the
components of a structure vector represent different numerical values and
similarly, for the letters appearing in the elements $\mathbf{abc}$ of the
standard basis of $\boldsymbol{\Lambda}$.

\medskip

In the course of the discussion in the paper we will be pointing out the
places at which this convection will actually be in force.

\section{The $G$-submodules $\mathcal{C}$ and $\mathcal{K}$}

\label{SecSubmodulesCK}

In this section we discuss two special $G$-submodules of $\boldsymbol{\Lambda%
}$, namely $\mathcal{C}$ and $\mathcal{K}$, which, among them, contain all
composition factors of $\boldsymbol{\Lambda}$.

\subsection{Defining conditions and bases}

\label{SubsecDefCondBases}

Convention~($\ddag$) will be in force for the whole of Subsection~\ref%
{SubsecDefCondBases}.

\medskip

The subset $\mathcal{C}$ of $\boldsymbol{\Lambda}$ is defined by the
requirement that $\boldsymbol{\lambda}=\Theta(\mathfrak{g})$ is a member of $%
\mathcal{C}$, precisely when $[u,v]=[v,u]$ for all $u,v\in V$, where $[,]$
denotes the product in the algebra $\mathfrak{g}$.

It follows that the conditions
\begin{gather*}
\lambda_{ijj}=\lambda_{jij}  \notag \\
\lambda_{ijk}=\lambda_{jik}
\end{gather*}
form a set of defining conditions for~$\mathcal{C}$. In particular $\mathcal{%
C}$ is a subspace of $\boldsymbol{\Lambda}$. Comparing with Definition~\ref%
{DefAction} and assuming that $[,]$ is commutative, we see that $%
[u,v]^{\prime}=g^{-1}[gu,gv]=g^{-1}[gv,gu]=[v,u]^{\prime}$. It follows that $%
\mathcal{C}$ is a $G$-submodule of~$\boldsymbol{\Lambda}$.

\medskip

The space $\mathcal{C}$ has the following set of structure vectors as a
basis:
\begin{equation*}
\begin{tabular}{cc}
vector & number \\
$\mathbf{iii}$ & $n$ \\
$\mathbf{iij}$ & $n(n-1)$ \\
$\mathbf{iji}+\mathbf{jii}$ & $n(n-1)$ \\
$\mathbf{ijk}+\mathbf{jik}$ & $\binom{n}{2}(n-2)$%
\end{tabular}%
\end{equation*}%
In particular, $\dim\mathcal{C}=n^{3}/2+n^{2}/2$. Note that in the last item
of the table above the distinct members $\mathbf{ijk}+\mathbf{jik}$ are
obtained by imposing the restriction $i<j$.

\medskip

The subset $\mathcal{K}$ of $\boldsymbol{\Lambda}$ is defined by the
requirement that $\boldsymbol{\lambda}=\Theta(\mathfrak{g})$ belongs to $%
\mathcal{K}$, precisely when $[v,v]=0$ for all $v\in V$.
By~\cite[Remark~2.7]{IvanovaPallikaros2019} and item (ii) before that,
the conditions
\begin{gather*}
\lambda_{iii}=0,\quad \lambda_{iij}=0  \notag \\
\lambda_{ijk}+\lambda_{jik}=0  \notag \\
\lambda_{iji}+\lambda_{jii}=0
\end{gather*}
form a set of of defining conditions for $\mathcal{K}$. As in the case of $%
\mathcal{C}$, it is again easy to observe that $\mathcal{K}$ is a $G$%
-submodule of $\boldsymbol{\Lambda}$. Moreover, $\mathcal{K}$ has the
following set of structure vectors as a basis:
\begin{equation*}
\begin{tabular}{cc}
vector & number \\
$\mathbf{iji}-\mathbf{jii}$ & $n(n-1)$ \\
$\mathbf{ijk}-\mathbf{jik}$ & $\binom{n}2(n-2)$%
\end{tabular}%
\end{equation*}%
So $\dim\mathcal{K}=\frac12n^3-\frac12n^2=n^3-\dim\mathcal{C}$.

\begin{remark}
\label{RemSkewAlg} (i) If $\mathrm{char}\,\mathbb{F}\ne2$, then $\mathcal{C}%
\cap\mathcal{K}=0$, so $\boldsymbol{\Lambda}=\mathcal{C}\oplus\mathcal{K}$.

(ii) If $\mathop{\rm char}\nolimits\mathbb{F}=2$, then $\mathcal{K}\subset
\mathcal{C}$. Also note that our proposed basis for $\mathcal{K}$ is
contained in our proposed basis for $\mathcal{C}$. In particular, the cosets
$\mathbf{iii}+\mathcal{K}$ and $\mathbf{iij}+\mathcal{K}$,
form a basis for~$\mathcal{C}/\mathcal{K}$.
\end{remark}

\subsection{The `opposite' algebra}

For an algebra $\mathfrak{h}$ with product $[,]$, the \textbf{opposite
algebra} $\widetilde{\mathfrak{h}}$ has product $\widetilde{[\,,\,]}$
defined by $\widetilde{[u,v]}=[v,u]$. If $\Theta(\mathfrak{h})=\boldsymbol{%
\mu}$ with $\boldsymbol{\mu}=(\mu_{ijk})$, we will write $\Theta(\widetilde{%
\mathfrak{h}})=\widetilde{\boldsymbol{\mu}}$ with $\widetilde{\boldsymbol{\mu%
}}=(\widetilde{\mu}_{ijk})$. Clearly $(\widetilde{\tilde{\boldsymbol{\mu}}})=%
\boldsymbol{\mu}$ and $\widetilde\mu_{ijk}=\mu_{jik}$ for all $i,j,k$.

Suppose now that $\mathfrak{g}\in\boldsymbol{A}$ has product $[,]$ and let $%
\boldsymbol{\lambda}=\Theta(\mathfrak{g})$. Suppose further that $g\in G$ is
the transition map from the standard basis $v_1,\ldots,v_n$ to the basis $%
v_1^{\prime},\ldots,v_n^{\prime}$ of $V$, so that $v_i^{\prime}=gv_i$ for $%
i=1,\ldots,n$. It is then easy to observe that for all $i$, $j$ and $k$, the
coefficient of $\mathbf{ijk}$ when we express either $(\widetilde{%
\boldsymbol{\lambda}})g$ or $(\widetilde{\boldsymbol{\lambda }g})$ as a
linear combination of the elements of the standard basis of $\boldsymbol{%
\Lambda}$, equals the coefficient of $v_k^{\prime}$ when we express $%
[v_j^{\prime},v_i^{\prime}]$ as a linear combination of the elements of the
basis $v_1^{\prime},\ldots,v_n^{\prime}$ of $V$. We have proved:

\begin{lemma}
We have that $(\widetilde{\boldsymbol{\lambda}})g=(\widetilde{\boldsymbol{%
\lambda }g})$ for all $\boldsymbol{\lambda}\in\boldsymbol{\Lambda}$ and for
all $g\in G$. Hence, the maps $\boldsymbol{\lambda}\mapsto\widetilde{%
\boldsymbol{\lambda}}$ and $\boldsymbol{\lambda}\mapsto\boldsymbol{\lambda}+%
\widetilde{\boldsymbol{\lambda}}$ from $\boldsymbol{\Lambda}$ to $%
\boldsymbol{\Lambda}$ are $G$-homomorphisms.
\end{lemma}

Writing $\widetilde{\mathcal{X}}=\{\widetilde{\boldsymbol{\lambda}}\colon
\boldsymbol{\lambda}\in\mathcal{X}\}$ for a subset $\mathcal{X}$ of $%
\boldsymbol{\Lambda}$ we see that $\widetilde{\mathcal{X}}$ is a $G$%
-submodule of $\boldsymbol{\Lambda}$ whenever $\mathcal{X}$ is a $G$%
-submodule of $\boldsymbol{\Lambda}$. Since $\boldsymbol{\lambda}=\widetilde{%
\boldsymbol{\lambda}}$ (resp., $\boldsymbol{\lambda}=-\widetilde{\boldsymbol{%
\lambda}}$) for each $\boldsymbol{\lambda}\in\mathcal{C}$ (resp., $%
\boldsymbol{\lambda}\in\mathcal{K}$) we see that $\mathcal{C}=\widetilde{%
\mathcal{C}}$ (resp., $\mathcal{K}=\widetilde{\mathcal{K}}$). Moreover,
we have that $\boldsymbol{\lambda}+\widetilde{\boldsymbol{\lambda}}\in%
\mathcal{C}$ and $\boldsymbol{\lambda}-\widetilde{\boldsymbol{\lambda}}\in%
\mathcal{K}$ for every $\boldsymbol{\lambda}\in\boldsymbol{\Lambda}$.

Suppose now that $\mathrm{char}\,\mathbb{F}=2$ and consider the map $%
\boldsymbol{\lambda}\mapsto\boldsymbol{\lambda}+\widetilde{\boldsymbol{%
\lambda}}\,( = \boldsymbol{\lambda }- \widetilde{\boldsymbol{\lambda}})$
from $\boldsymbol{\Lambda}$ to~$\boldsymbol{\Lambda}$. This is a $G$%
-homomorphism having $\mathcal{C}$ as its kernel and $\mathcal{K}$ as its
image, as is easily seen from the defining conditions for $\mathcal{C}$ and $%
\mathcal{K}$. Hence, in characteristic 2, we have a filtration $0\subset
\mathcal{K}\subset\mathcal{C}\subset\boldsymbol{\Lambda}$ with $\boldsymbol{%
\Lambda}/\mathcal{C}$ being $G$-isomorphic to $\mathcal{K}$.

\section{Adjoint trace form and unimodular algebras}

\label{SectAdjTrace}

Following \cite[Section 4.1]{IvanovaPallikaros2019}, we define the adjoint
map for an algebra $\mathfrak{g}$ to be $\mathrm{ad}_{u}:v\mapsto \lbrack
u,v]$. With $\boldsymbol{\lambda }\,(=(\lambda_{ijk}))=\Theta (\mathfrak{g})$%
, we set up the \textbf{adjoint trace form}, the pairing $\mathrm{tr}(%
\boldsymbol{\lambda },u)=\mathrm{tr}(\mathrm{ad}_{u})$. A direct computation
shows that if $u=\sum \xi _{i}v_{i}$, a linear combination of the elements
of the standard basis $v_1,\ldots,v_n$ of $V$, then
\begin{equation}
\mathrm{tr}(\mathrm{ad}_{u})=\sum_{i,j}\xi _{i}\lambda _{ijj}.
\label{traceformula}
\end{equation}

\begin{lemma}
\label{TrPairing}If $g\in G$, then $\mathrm{tr}(\boldsymbol{\lambda }g,u)=%
\mathrm{tr}(\boldsymbol{\lambda },gu)$.
\end{lemma}

\begin{proof}
Let $\mathfrak{g}^{\prime }=\Theta ^{-1}(\boldsymbol{\lambda }g)$.
Then $\mathrm{%
tr}(\boldsymbol{\lambda }g,u)$ is the trace of the map $v\mapsto \lbrack
u,v]^{\prime }$. But $[u,v]^{\prime }=g^{-1}[gu,gv]$
(see Definition~\ref{DefAction}), and the
map is the composition $v\mapsto gv\mapsto \lbrack gu,gv]\mapsto
g^{-1}[gu,gv]$. This composition is the conjugate by $g$ of the middle map $%
w\mapsto \lbrack gu,w]$. So $\mathrm{tr}(\boldsymbol{\lambda }g,u)=\mathrm{tr%
}(\boldsymbol{\lambda },gu)$, as claimed.
\end{proof}

The pairing $\mathrm{tr}(\boldsymbol{\lambda },u)$ is thus bilinear and $G$%
-invariant (left action on $V$, right on $\boldsymbol{\Lambda }$). Define $%
\mathrm{tr}_{\boldsymbol{\lambda }}$ to be the member of $\widehat{V}$ given
by $u\mapsto \mathrm{tr}(\boldsymbol{\lambda },u)$. Since%
\begin{equation*}
(\mathrm{tr}_{\boldsymbol{\lambda }}g)(u)=\mathrm{tr}_{\boldsymbol{\lambda }%
}(gu)=\mathrm{tr}(\boldsymbol{\lambda },gu)=\mathrm{tr}(\boldsymbol{\lambda }%
g,u)=\mathrm{tr}_{\boldsymbol{\lambda }g}(u),
\end{equation*}
$\mathrm{tr}:\boldsymbol{\lambda }\mapsto \mathrm{tr}_{\boldsymbol{\lambda }%
} $ from $\boldsymbol{\Lambda }$ to $\widehat{V}$ is a $G$-homomorphism.
Recall that $\widehat{v_{1}},\ldots ,\widehat{v_{n}}$ is the dual basis of $%
v_{1},\ldots ,v_{n}$: $\widehat{v_{i}}(v_{j})=\delta _{ij}$. Then (\ref%
{traceformula}) gives
\begin{equation}
\mathrm{tr}_{\boldsymbol{\lambda }}=\sum_{i}\left( \sum_{j}\lambda
_{ijj}\right) \widehat{v_{i}}.  \label{trbydual}
\end{equation}%
In particular, $\mathrm{tr}_{\mathbf{iii}}=\widehat{v_{i}}$. Thus the map $%
\boldsymbol{\lambda }\mapsto \mathrm{tr}_{\boldsymbol{\lambda }}$ is a $G$%
-homomorphism of $\boldsymbol{\Lambda }$ onto $\widehat{V}$. We denote its
kernel by $\mathcal{T}$. (The members of $\Theta^{-1}(\mathcal{T})$ are
known as unimodular algebras.) We have:

\begin{proposition}
\label{PropL/TIsomToHatV} $\boldsymbol{\Lambda }/\mathcal{T}$ is $G$%
-isomorphic to $\widehat{V}$. Thus $\mathcal{T}$ has codimension $n$ in $%
\boldsymbol{\Lambda }$.
\end{proposition}

In \cite[Definition 4.13]{IvanovaPallikaros2019}, the $G$-submodule $%
\mathcal{U}$ is defined to be $\mathcal{K}\cap \mathcal{T}$.
Equation (\ref{trbydual}) gives $\mathrm{tr}_{\mathbf{ijj}-\mathbf{jij}}=%
\widehat{v_{i}}$, for $i\ne j$, so the map $\boldsymbol{\lambda }\mapsto
\mathrm{tr}_{\boldsymbol{\lambda }}$ from $\mathcal{K}$ to $\widehat{V}$ is
also surjective in view of the fact that the structure vectors $\mathbf{ijj}-%
\mathbf{jij}$ belong to $\mathcal{K}$, as we have seen in Section~\ref%
{SecSubmodulesCK}. Thus $\mathcal{K}/\mathcal{U}\backsimeq \widehat{V}$,
too, verifying that $\dim \mathcal{U}=(n^{3}-n^{2})/2-n$.

Next, we restrict the map $\mathrm{tr}$ to the submodule $\mathcal{C}$ of $%
\boldsymbol{\Lambda}$. Let $\mathcal{N}$ be the kernel of this restriction.
Clearly, $\mathcal{N}=\mathcal{C}\cap\mathcal{T}$ and $\mathcal{N}$ is a $G$%
-submodule of $\boldsymbol{\Lambda}$.
Since, as we have seen, $\mathrm{tr}_{\mathbf{iii}}=\widehat{v_{i}}$ and $%
\mathbf{iii}\in\mathcal{C}$ for all $i$, this restricted map is also
surjective. It follows that $\mathcal{C}/\mathcal{N}$ and $\widehat V$ are $%
G $-isomorphic. Summing up:

\begin{proposition}
\label{PropIsomDualV} $\mathcal{K}/\mathcal{U}$, $\mathcal{C}/\mathcal{N}$
and $\widehat V$ are $G$-isomorphic.
\end{proposition}

The members of $\mathcal{N}$ are the structure vectors $\boldsymbol{\lambda }
$ in $\mathcal{C}$ for which $\sum \lambda _{ijj}=0$.
So $\mathcal{N}$ has basis
(assuming that Convention~($\ddag$) is in force for the following table)
\begin{equation}  \label{basisN}
\begin{tabular}{cc}
vector & number \\
$\mathbf{ijk}+\mathbf{jik}$ & $n(n-1)(n-2)/2$ \\
$\mathbf{iij}$ & $n(n-1)$ \\
$\mathbf{ijj}+\mathbf{jij}-\mathbf{iii}$ & $n(n-1)$%
\end{tabular}%
\end{equation}%
giving $\dim \mathcal{N}=n^{3}/2+n^{2}/2-n$, in line with the $G$%
-isomorphism $\mathcal{C}/\mathcal{N}\backsimeq \widehat{V}$.

Imitating the discussion at the beginning of this section, let $\widetilde{%
\mathrm{tr}}(\boldsymbol{\lambda},u)=\mathrm{tr}(v\mapsto[v,u])$ be the
opposite trace map, and define $\widetilde{\mathrm{tr}}_{\boldsymbol{\lambda}%
}$ to be the member of $\widehat V$ given by $u\mapsto\widetilde{\mathrm{tr}}%
(\boldsymbol{\lambda},u)$. Note that $\widetilde{\mathrm{tr}}(\boldsymbol{%
\lambda},u)=\mathrm{tr}(v\mapsto\widetilde{[u, v]})=\mathrm{tr}(\widetilde{%
\boldsymbol{\lambda}},u)$. It follows that $\widetilde{\mathrm{tr}}\colon%
\boldsymbol{\lambda}\mapsto\widetilde{\mathrm{tr}}_{\boldsymbol{\lambda}}\,(=%
\mathrm{tr}_{\widetilde{\boldsymbol{\lambda}}})$ is a surjective $G$%
-homomorphism from $\boldsymbol{\Lambda}$ to $\widehat V$ and, moreover,
\begin{equation}  \label{trlabdatilde}
\widetilde{\mathrm{tr}}_{\boldsymbol{\lambda}}=\sum_{i}\left(
\sum_{j}\lambda _{jij}\right) \widehat{v_{i}}.
\end{equation}
Clearly, $\ker\widetilde{\mathrm{tr}}=\{\boldsymbol{\lambda}\colon
\widetilde{\boldsymbol{\lambda}}\in\mathcal{T}\}=\widetilde{\mathcal{T}}$.
In particular, $\boldsymbol{\Lambda}/\widetilde{\mathcal{T}}$ and $\widehat
V $ are $G$-isomorphic.

We now restrict the map $\mathop{\mathrm{tr}}\nolimits$ to the submodule $%
\widetilde{\mathcal{T}}$ of $\boldsymbol{\Lambda}$. Clearly, $\mathcal{T}\cap%
\widetilde{\mathcal{T}}$ is the kernel of this restriction. Let $\boldsymbol{%
\mu}=\mathbf{122}+\mathbf{212}-\mathbf{313}$. It is easy to check that $%
\widetilde{\boldsymbol{\mu}}\in\mathcal{T}$ (so $\boldsymbol{\mu}\in%
\widetilde{\mathcal{T}}$) and that $\mathop{\mathrm{tr}}\nolimits_{%
\boldsymbol{\mu}}=\hat v_1$. Since $\widehat V$ is an irreducible $G$%
-module, we have:

\begin{proposition}
$\widehat V$, $\widetilde{\mathcal{T}}/(\mathcal{T}\cap\widetilde{\mathcal{T}%
})$ (and, by similar argument, ${\mathcal{T}}/(\mathcal{T}\cap\widetilde{%
\mathcal{T}})$) are $G$-isomorphic. In particular, $\dim(\mathcal{T}\cap%
\widetilde{\mathcal{T}})=\dim\boldsymbol{\Lambda}-2n=\dim\mathcal{U}+\dim%
\mathcal{N}$.
\end{proposition}

It is easy to observe that $\mathcal{U}$ and $\mathcal{N}$ are both
contained in $\mathcal{T}\cap\widetilde{\mathcal{T}}$. If $\mathop{\rm char}%
\nolimits\mathbb{F}\ne2$, then $\mathcal{T}\cap\widetilde{\mathcal{T}}=%
\mathcal{U}\oplus\mathcal{N}$, since $\mathcal{U}\cap\mathcal{N}=0$ in this
case. If $\mathop{\rm char}\nolimits\mathbb{F}=2$, the map $\boldsymbol{%
\lambda}\mapsto\boldsymbol{\lambda}+\widetilde{\boldsymbol{\lambda}}$
defines a $G$-homomorphism from $\mathcal{T}\cap\widetilde{\mathcal{T}}$ to $%
\boldsymbol{\Lambda}$. Comparing with the discussion in Section~\ref%
{SecSubmodulesCK}, we see that the kernel of this map is $(\mathcal{T}\cap%
\widetilde{\mathcal{T}})\cap\mathcal{C }=\mathcal{N}$. Moreover, since $%
\boldsymbol{\lambda}+\widetilde{\boldsymbol{\lambda}}\in\mathcal{T}\cap%
\widetilde{\mathcal{T}}$ for all $\boldsymbol{\lambda}\in\mathcal{T}\cap%
\widetilde{\mathcal{T}}$ we get that the image of this map is contained in $(%
\mathcal{T}\cap\widetilde{\mathcal{T}})\cap\mathcal{K}=\mathcal{U}$.
Finally, comparing dimensions we conclude that this image in fact equals $%
\mathcal{U}$, so in characteristic 2 we again have that $(\mathcal{T}\cap%
\widetilde{\mathcal{T}})/\mathcal{N}$ and $\mathcal{U}$ are $G$-isomorphic.

\medskip

Consider now the filtration $0\subset \mathcal{N}\subset\mathcal{T}\cap%
\widetilde{\mathcal{T}}\subset {\mathcal{T}}\subset\boldsymbol{\Lambda}$
with no restriction on the field $\mathbb{F}$. We have shown that the last
two factors are $G$-isomorphic to $\widehat V$, whereas $(\mathcal{T}\cap%
\widetilde{\mathcal{T}})/\mathcal{N}$ is $G$-isomorphic to $\mathcal{U}$.
Note also that in characteristic 2 we also have $\mathcal{U}\subset \mathcal{%
N}$, since $\mathcal{K}\subset\mathcal{C}$.

\medskip

It is convenient at this point to introduce the elements $\boldsymbol{\eta}$
and $\boldsymbol{\delta}$ of $\boldsymbol{\Lambda}$ where $\boldsymbol{\eta}=%
\mathbf{123}-\mathbf{213}$ and $\boldsymbol{\delta}=\mathbf{112}$.
Thus $\boldsymbol{\eta}\in\mathcal{U}$ and $\boldsymbol{\delta}\in\mathcal{N}
$.
More can be shown:

\begin{remark}
\label{RemUeta} In~\cite[Lemma~4.14]{IvanovaPallikaros2019} it was shown
that $\mathcal{U}=\boldsymbol{\eta}(\mathbb{F} G)$ under the running
assumption that $\mathbb{F}$ is infinite, however the proof given there goes
through without any change in the case of an arbitrary field $\mathbb{F}$.
Hence, $\mathcal{U}=\boldsymbol{\eta}(\mathbb{F} G)$ for any field $\mathbb{F%
}$.
\end{remark}

\begin{proposition}
\label{PropNdelta} Suppose $|\mathbb{F}|>2$. Then \label{N}$\mathcal{N}=%
\boldsymbol{\delta }(\mathbb{F}G)$.
\end{proposition}

\begin{proof}
We use the formula~\eqref{basic} with various choices of $g\in G$ to
produce other members of $\boldsymbol{\delta }(\mathbb{F}G)$. By basis
permutations, we get that $\boldsymbol{\delta }(\mathbb{F}G)$ contains all
structure vectors $\mathbf{iij}$, with $j\neq i$.
An immediate consequence of formula~\eqref{basic} is that
\begin{equation}
\mathbf{abc}g=\sum_{i,j,k}g_{ai}g_{bj}g_{kc}^{(-1)}\mathbf{ijk}.
\label{TripleChange}
\end{equation}
In our case,~\eqref{TripleChange} reads%
\begin{eqnarray*}
\mathbf{112}g &=&\sum_{i,j,k}g_{1i}g_{1j}g_{k2}^{(-1)}\mathbf{ijk} \\
&=&\sum_{i,k}g_{1i}^{2}g_{k2}^{(-1)}\mathbf{iik}%
+\sum_{i<j,k}g_{1i}g_{1j}g_{k2}^{(-1)}(\mathbf{ijk}+\mathbf{jik}).
\end{eqnarray*}%
First take $g\in G$ with
\begin{equation*}
[g]=%
\begin{bmatrix}
1 & 1 & 0 & 0 \\
0 & 0 & 1 & 0 \\
0 & 1 & 0 & 0 \\
0 & 0 & 0 & I_{n-3}%
\end{bmatrix}%
,\quad [g^{-1}]=%
\begin{bmatrix}
1 & 0 & -1 & 0 \\
0 & 0 & 1 & 0 \\
0 & 1 & 0 & 0 \\
0 & 0 & 0 & I_{n-3}%
\end{bmatrix}%
.
\end{equation*}%
Then $\mathbf{112}g=\mathbf{113}+\mathbf{223}+(\mathbf{123}+\mathbf{213})$.
So $\mathbf{123}+\mathbf{213}\in \boldsymbol{\delta }(\mathbb{F}G)$, and
then by permutations, all $\mathbf{ijk}+\mathbf{jik}$ (for distinct $i,j,k$) belong to $\boldsymbol{\delta }(%
\mathbb{F}G)$.
Now take $g\in G$ with
\begin{equation*}
[g]=%
\begin{bmatrix}
1 & 1 & 0 \\
0 & 1 & 0 \\
0 & 0 & I_{n-2}%
\end{bmatrix},%
\quad [g^{-1}]=%
\begin{bmatrix}
1 & -1 & 0 \\
0 & 1 & 0 \\
0 & 0 & I_{n-2}%
\end{bmatrix}%
.
\end{equation*}%
Then%
\begin{eqnarray*}
\mathbf{112}g &=&-\mathbf{111}+\mathbf{112}-\mathbf{221}+\mathbf{222}-(%
\mathbf{121}+\mathbf{211})+(\mathbf{122+212}) \\
&=&-(\mathbf{121}+\mathbf{211}-\mathbf{222})+(\mathbf{122+212}-\mathbf{111})+%
\mathbf{112}-\mathbf{221}.
\end{eqnarray*}%
Hence $-(\mathbf{121}+\mathbf{211}-\mathbf{222})+(\mathbf{122}+\mathbf{212}-\mathbf{111})\in\boldsymbol\delta(\F G)$.

Finally take $g\in{\rm GL}(V)$ with
$[g]=\begin{bmatrix}
\alpha & 0 \\
0 &  I_{n-1}%
\end{bmatrix}$,
where $\alpha\in\F-\{0,1\}$.
Then $(-(\mathbf{121}+\mathbf{211}-\mathbf{222})+(\mathbf{122}+\mathbf{212}-\mathbf{111}))g= \alpha(\mathbf{122}+\mathbf{212}-\mathbf{111})-(\mathbf{121}+\mathbf{211}-\mathbf{222})\in\boldsymbol\delta(\F G)$.
Subtracting, shows that $(1-\alpha)(\mathbf{122}+\mathbf{212}-\mathbf{111})\in\boldsymbol\delta(\F G)$.
Hence, $\mathbf{122}+\mathbf{212}-\mathbf{111}\in\boldsymbol\delta(\F G)$, since $\alpha\ne1$.
We conclude that all $\mathbf{iji}+\mathbf{jii}-\mathbf{jjj}$ (for distinct $i,j$) belong to $\boldsymbol\delta(\F G)$.
Thus from Table~\ref{basisN}, the basis elements of $\mathcal N$ are all present and $\boldsymbol\delta(\F G)=\mathcal N$.
\end{proof}

The submodules $\mathcal{U}$ and $\mathcal{N}$, and their generators $%
\boldsymbol{\eta}$ and $\boldsymbol{\delta}$, will play an important part in
understanding the $GL(V)$-structure of $\boldsymbol{\Lambda}$ and the
composition series of some of its important $G$-submodules as we will see in
subsequent sections.
First, we will need to determine the intersection of $\mathcal{U}$ and $%
\mathcal{N}$ with two special $G$-submodules of $\boldsymbol{\Lambda}$,
namely $\mathcal{M}^{\ast }$ and $\mathcal{M}^{\ast \ast}$, the structure of
which we discuss in the next two sections.

\section{The structure of $\mathcal{M}^{\ast }$}

\label{SecM*}

\subsection{Defining conditions}

\label{SubsecDefConds}

We define $\mathcal{M}^{\ast }$ to be the set of structure vectors $%
\boldsymbol{\lambda}$ whose corresponding algebras $\Theta^{-1}(\boldsymbol{%
\lambda })$ satisfy the condition $[u,v]\in \mathbb{F}$-$\mathrm{sp}(u,v)$,
the $\mathbb{F}$-span of $u$ and $v$. Clearly $\mathcal{M}^{\ast }$ is a $G$%
-submodule of $\boldsymbol{\Lambda}$. We first wish to bound the dimension
of $\mathcal{M}^{\ast }$. Recall that $v_{1},\ldots ,v_{n}$ is the standard
basis for $V$.

\medskip

Convention~($\ddag$) will be in force for the whole of the Subsection~\ref%
{SubsecDefConds}.

\begin{lemma}
We have $\dim \mathcal{M}^{\ast }\leq 2n$.
\end{lemma}

\begin{proof}
Since $[v_{i},v_{i}]\in \mathbb{F}v_{i}$, it must be
that $\lambda _{iij}=0$ (for all $j\neq i$). Similarly, $[v_{i},v_{j}]\in \mathbb{F}$-$\mathrm{sp}(v_{i},v_{j})$
implies that $\lambda _{ijk}=0$. So far we have $%
n(n-1)+n(n-1)(n-2)=n^{3}-2n^{2}+n$ independent conditions on the structure
constants. Next,%
\begin{eqnarray*}
\lbrack v_{i},v_{j}+v_{k}] &=&[v_{i},v_{j}]+[v_{i},v_{k}] \\
&=&\lambda _{iji}v_{i}+\lambda _{ijj}v_{j}+\lambda
_{iki}v_{i}+\lambda _{ikk}v_{k}.
\end{eqnarray*}%
As the result must be $\xi v_{i}+\eta (v_{j}+v_{k})$
for some $\xi ,\eta $, we need $\lambda _{ijj}=\lambda _{ikk}$ for all
choices. Similarly, $\lambda _{jij}=\lambda _{kik}$.
Thus we may write $%
\delta _{i}=\lambda _{ijj}$ and $\alpha _{i}=\lambda _{jij}$. The
computation creates $2\times n(n-2)=2n^{2}-4n$ more conditions, making $\dim
\mathcal{M}^{\ast }\leq 3n$. Finally, we have%
\begin{eqnarray*}
\lbrack v_{i}+v_{j},v_{i}+v_{k}]
&=&[v_{i},v_{i}]+[v_{i},v_{k}]+[v_{j},v_{i}]+[v_{j},v_{k}] \\
&=&\lambda _{iii}v_{i}+\lambda _{iki}v_{i}+\lambda
_{ikk}v_{k}+\lambda _{jii}v_{i} \\
&&+\lambda _{jij}v_{j}+\lambda _{jkj}v_{j}+\lambda
_{jkk}v_{k},
\end{eqnarray*}%
and this must be $\xi (v_{i}+v_{j})+\eta (v_{i}+v_{k})$ for some $\xi ,\eta $. So%
\begin{eqnarray*}
\xi +\eta &=&\lambda _{iii}+\lambda _{iki}+\lambda _{jii} \\
\xi &=&\lambda _{jij}+\lambda _{jkj} \\
\eta &=&\lambda _{ikk}+\lambda _{jkk}.
\end{eqnarray*}%
Then%
\[
\lambda _{iii}+\lambda _{iki}+\lambda _{jii}=\lambda _{jij}+\lambda
_{jkj}+\lambda _{ikk}+\lambda _{jkk},
\]%
making%
\begin{eqnarray*}
\lambda _{iii} &=&\lambda _{jij}+\lambda _{jkj}+\lambda _{ikk}+\lambda
_{jkk}-\lambda _{iki}-\lambda _{jii} \\
&=&\alpha _{i}+\alpha_{k}+\delta_{i}+\delta_{j}-\alpha_{k}-\delta_{j} \\
&=&\alpha _{i}+\delta _{i}.
\end{eqnarray*}%
This gives a further $n$ conditions and the desired result: $\dim \mathcal{M}%
^{\ast }\leq 2n$. Here are the relations for $\mathcal{M}^{\ast }$ again:%
\begin{eqnarray}
\lambda _{iij} &=&0,\quad \lambda _{ijk}=0 \notag\\
\lambda _{ijj} &=&\lambda _{ikk},\quad \lambda _{jij}=\lambda _{kik} \label{cond*} \\
\lambda _{iii} &=&\lambda _{ijj}+\lambda _{jij}. \notag
\end{eqnarray}
\end{proof}

Now let $\alpha $ and $\delta $ be two linear functionals on $V$ and define
the the algebra $\mathfrak{m}_{\alpha ,\delta }$ with structure vector $%
\boldsymbol{\mu }_{\alpha ,\delta }=\Theta (\mathfrak{m}_{\alpha ,\delta })$
by the multiplication rule $[u,v]=\alpha (v)u+\delta (u)v$. Evidently $%
\boldsymbol{\mu }_{\alpha ,\delta }\in \mathcal{M}^{\ast }$. Since the set
of such algebras is a $2n$-dimensional space, they must make up $\Theta
^{-1}(\mathcal{M}^{\ast })$:

\begin{proposition}
The dimension of $\mathcal{M}^{\ast }$ is $2n$, and its members are the
structure vectors $\boldsymbol{\mu }_{\alpha ,\delta }$.
\end{proposition}

Alternatively, it is easy to check directly that the conditions~\eqref{cond*}
are also sufficient for the structure vector $\boldsymbol{\lambda}$ to be a
member of $\mathcal{M}^{\ast}$ (and hence they constitute a set of defining
conditions for $\mathcal{M}^{\ast}$). For this, let $u=\sum_i\xi_iv_i$ and $%
v=\sum_i\xi_i^{\prime}v_i$ and assume conditions~\eqref{cond*} hold. On
setting $\alpha_i=\lambda_{jij}$ and $\delta_i=\lambda_{ijj}$ as above, we
get that in $\Theta^{-1}(\boldsymbol{\lambda})$ the coefficient of $v_k$ in
the expression of $[u,v]$ as a linear combination of our standard basis $%
v_1,\ldots, v_n$, equals $\xi_k^{\prime}(\sum_i\xi_i\delta_i)+\xi_k(\sum_i%
\xi_i^{\prime}\alpha_i)$. Thus $[u,v]=(\sum_i\xi_i^{\prime}\alpha_i)u+(%
\sum_i\xi_i\delta_i)v$. In particular, we have $\alpha(v)=\sum_i\xi_i^{%
\prime}\alpha_i$ and $\delta(u)=\sum_i\xi_i\delta_i$.

\subsection{Action of $G$ on $\mathcal{M}^{\ast }$ and structure vectors}

If $\mathfrak{g}$ is an algebra and $g\in G$, recall that then the image $%
\mathfrak{g}g=\mathfrak{g}^{\prime }$ has product given by $[u,v]^{\prime
}=g^{-1}[gu,gv]$.
For $\mathfrak{m}_{\alpha ,\delta }g$ we have the product%
\begin{equation*}
\lbrack u,v]^{\prime }=g^{-1}[gu,gv]=g^{-1}(\alpha (gv)gu+\delta
(gu)gv)=\alpha (gv)u+\delta (gu)v.
\end{equation*}%
Thus $\mathfrak{m}_{\alpha ,\delta }g=\mathfrak{m}_{\alpha g,\delta g}$, so
that $\boldsymbol{\mu }_{\alpha ,\delta }g=\boldsymbol{\mu }_{\alpha
g,\delta g}$. In particular, $\mathcal{M}^{\ast }$ is isomorphic to $%
\widehat{V}\oplus \widehat{V}$ as a $G$-module, one isomorphism being $%
\boldsymbol{\mu }_{\alpha ,\delta }\mapsto (\alpha ,\delta )$. The
transitivity properties of $G$ on $\widehat{V}$, which parallel those on $V$%
, show that $\widehat{V}$ is irreducible, and then $\mathcal{M}^{\ast }$ is
completely reducible. Moreover, if $\alpha $ and $\delta $ are independent,
then $(\alpha ,\delta )\mathbb{F}G=\widehat{V}\oplus \widehat{V}$. If $%
\alpha $ and $\delta $ are not independent and not both 0, then $(\alpha
,\delta )\mathbb{F}G$ is an irreducible submodule. Suppose that for some
nonzero member $P=(P_{\alpha },P_{\delta })$ of $\mathbb{F}^{2}$, $P_{\delta
}\alpha -P_{\alpha }\delta =0$. Then%
\begin{equation*}
(\alpha ,\delta )\mathbb{F}G=\left\{ (P_{\alpha }\theta ,P_{\delta }\theta
)|\theta \in \widehat{V}\right\} .
\end{equation*}%
We denote the corresponding submodule of $\mathcal{M}^{\ast }$ by $\mathcal{M%
}_{P}^{\ast }$. It follows that $\mathcal{M}_{P}^{\ast }$ is an irreducible
submodule which is $G$-isomorphic to $\widehat{V}$. The irreducible
submodules of $\mathcal{M}^{\ast }$ are the $\mathcal{M}_{P}^{\ast }$, $P$
running over a set of representatives of the one-dimensional subspaces of $%
\mathbb{F}^{2}$ (the projective line over $\mathbb{F}$).

For $\mathfrak{m}_{\alpha ,\delta }$ we have $[u,u]=(\alpha (u)+\delta (u))u$%
. 
The subspace $\mathcal{K}$ 
consists of the $\boldsymbol{\lambda }$ for which $\alpha +\delta =0$. Thus $%
\mathcal{K}\cap \mathcal{M}^{\ast }=\mathcal{M}_{(1,-1)}^{\ast }$. As to the
adjoint trace form $\mathrm{tr}(\boldsymbol{\mu }_{\alpha ,\delta },v)=%
\mathrm{tr}(u\mapsto \lbrack v,u])$, we have that the trace of the map $%
u\mapsto \lbrack v,u]$ is the sum of the traces of the two maps $u\mapsto
\alpha (u)v$ and $u\mapsto \delta (v)u$. These are respectively $\alpha (v)$
and $n\delta (v)$. So $\mathrm{tr}(\boldsymbol{\mu }_{\alpha ,\delta
},v)=\alpha (v)+n\delta (v)$. It follows that $\mathcal{T}\cap \mathcal{M}%
^{\ast }=\mathcal{M}_{(-n,1)}^{\ast }$. Similarly, for the opposite adjoint
trace form $\widetilde{\mathrm{tr}}(\boldsymbol{\mu }_{\alpha ,\delta },v)=%
\mathrm{tr}(u\mapsto \lbrack u,v])$, we get $\widetilde{\mathrm{tr}}(%
\boldsymbol{\mu }_{\alpha ,\delta },v)=n\alpha (v)+\delta (v)$, and $%
\widetilde{\mathcal{T}}\cap \mathcal{M}^{\ast }=\mathcal{M}_{(1,-n)}^{\ast }$%
. In particular, $\boldsymbol{\mu }_{\alpha ,\delta }\in \mathcal{U}\cap
\mathcal{M}^{\ast }$ only when both $\delta =-\alpha $ and $-n\delta =\alpha
$. That is, we need $(n-1)\delta =0$. So if $\mathrm{char}\,\mathbb{F}$ does
not divide $n-1$, then $\mathcal{U}\cap \mathcal{M}^{\ast }=0$. But if it
does, then $\mathcal{U}\cap \mathcal{M}^{\ast }=\mathcal{M}_{(1,-1)}^{\ast }$%
. Here is a summary of these intersections:

\begin{proposition}
\label{PropIntersecM*} We have the following intersections with $\mathcal{M}%
^{\ast }$:%
\begin{equation*}
\begin{tabular}{ll}
$\mathcal{C}\cap \mathcal{M}^{\ast }$ & $\mathcal{M}_{(1,1)}^{\ast }$ \\
$\mathcal{K}\cap \mathcal{M}^{\ast }$ & $\mathcal{M}_{(1,-1)}^{\ast }$ \\
$\mathcal{T}\cap \mathcal{M}^{\ast }$ & $\mathcal{M}_{(-n,1)}^{\ast }$ \\
$\widetilde{\mathcal{T}}\cap \mathcal{M}^{\ast }$ & $\mathcal{M}%
_{(1,-n)}^{\ast }$ \\
$\mathcal{U}\cap \mathcal{M}^{\ast }$ & $\left\{
\begin{array}{c}
0,\quad \mathrm{char}\mathbb{F}\nmid n-1 \\
\mathcal{M}_{(1,-1)}^{\ast },\quad \mathrm{char}\mathbb{F}\mid n-1%
\end{array}%
\right. $ \\
$\mathcal{N\cap M}^{\ast }$ & $\left\{
\begin{array}{c}
0,\quad \mathrm{char}\mathbb{F}\nmid n+1 \\
\mathcal{M}_{(1,1)}^{\ast },\quad \mathrm{char}\mathbb{F}\mid n+1%
\end{array}%
\right. $%
\end{tabular}%
\end{equation*}
\end{proposition}

For structure vectors, let $\alpha =\zeta \widehat{v_{a}}$ and $\delta =\eta
\widehat{v_{d}}$. Then in $\Theta ^{-1}(\boldsymbol{\mu }_{\alpha ,\delta })$%
,%
\begin{eqnarray*}
\lbrack v_{i},v_{j}] &=&\zeta \widehat{v_{a}}(v_{j})v_{i}+\eta \widehat{v_{d}%
}(v_{i})v_{j} \\
&=&\zeta \delta _{aj}v_{i}+\eta \delta _{di}v_{j}.
\end{eqnarray*}%
So for $\boldsymbol{\mu }_{\alpha ,\delta }$, $\lambda _{iji}=\zeta \delta
_{aj}$ and $\lambda _{ijj}=\eta \delta _{di}$ when $i\ne j$. The only other
non-zero components are $\lambda_{aaa}=\zeta$ and $\lambda_{ddd}=\eta$ if $%
a\ne d$, and $\lambda_{aaa}\,(=\lambda_{ddd})=\zeta+\eta$ if $a=d$. It
follows that $\boldsymbol{\mu }_{\alpha ,\delta }=\zeta \sum \mathbf{iai}%
+\eta \sum \mathbf{djj}$, the sums unrestricted ($i=a$ and $j=d$ also
allowed). In particular, the sums $\sum \mathbf{iai}$ and $\sum \mathbf{djj}$
form a basis for $\mathcal{M}^{\ast }$.

\begin{remark}
From the description of the elements of $\mathcal{M}^{\ast}$ we
have obtained in this section, we can easily deduce that the Zariski-closure of the $G$-orbit of any nonzero element of $\mathcal M^{\ast}$ necessarily contains one of the $\mathcal M^{\ast}_P$'s whenever $\F$ is algebraically closed (compare~%
\cite[Lemma~5.5]{IvanovaPallikaros2019}).
For this, let $T$ be the subgroup of $G$ consisting of precisely
those $g\in G$ such that $[g]$ is diagonal. In view of~\cite[Lemma~3.2.3 and
Theorem~3.4.2]{Geck2003} it is enough to show that whenever $\boldsymbol{%
\lambda}\in\mathcal{M}^{\ast}$ satisfies $\boldsymbol{\lambda }t=\beta(t)%
\boldsymbol{\lambda}$, with $\beta(t)\in\mathbb{F}$, for all $t\in T$, then $%
\boldsymbol{\lambda}$ necessarily belongs to $\cup\mathcal{M}^{\ast}_P$. Set
$\boldsymbol\varepsilon_a=\sum\mathbf{iai}$ and $\widetilde{\boldsymbol\varepsilon}_a=\sum\mathbf{aii%
}$ (the sums unrestricted as above) for $1\le a\le n$, and let $\boldsymbol{%
\mu}\in\mathcal{M}^{\ast}$. Then $\boldsymbol{\mu}=\sum_a(\xi_a%
\boldsymbol\varepsilon_a+\xi_a^{\prime}\widetilde{\boldsymbol\varepsilon}_a)$ for some $\xi_a$, $%
\xi_a^{\prime}\in\mathbb{F}$ and $\boldsymbol{\mu }t=\sum_at_{aa}(\xi_a%
\boldsymbol\varepsilon_a+\xi_a^{\prime}\widetilde{\boldsymbol\varepsilon}_a)$. For this last sum to
be equal to $\beta(t)\boldsymbol{\mu}$ for all $t\in T$, the $\xi_a$, $%
\xi_a^{\prime}$ must be all 0 except possibly $\xi_b$ and $\xi_b^{\prime}$
for some $b$ with $1\le b\le n$. It follows that $\boldsymbol{\mu}=%
\boldsymbol{\mu}_{\alpha,\delta}$ where $\alpha=\xi_b\hat v_b$ and $%
\delta=\xi_b^{\prime}\hat v_b$. Thus $\boldsymbol{\mu}\in\cup\mathcal{M}%
^{\ast}_P$ as required.
\end{remark}

In fact more can be shown: 
Now let $\F$ be an arbitrary infinite field and let $\boldsymbol\mu=\boldsymbol\mu_{\alpha,\delta}\in\mathcal M^{\ast}$ with $\alpha,\delta$ linearly independent.
Given $\alpha^{\prime},\delta^{\prime}\in\widehat V$ with $\alpha^{\prime},\delta^{\prime}$ also linearly independent, there exists $g\in G$ such that $\alpha^{\prime}=\alpha g$ and $\delta^{\prime}=\delta g$ so $\boldsymbol\mu g=\boldsymbol\mu_{\alpha^{\prime},\delta^{\prime}}$.
It follows that  $\boldsymbol\mu G=\mathcal M^{\ast}-\cup\mathcal M^{\ast}_P$.
Moreover, the Zariski-closure of $\boldsymbol\mu G$, denoted by $\overline{\boldsymbol\mu G}$, is the whole of $\mathcal M^{\ast}$.
For this, first observe that an arbitrary submodule of $\mathcal M^{\ast}$ of the form $\mathcal M^{\ast}_P$ can be described as $\mathcal M^{\ast}_P=\F$-sp$(\{\xi\boldsymbol\varepsilon_i+\xi^{\prime}\widetilde{\boldsymbol\varepsilon}_i\colon 1\le i\le n\}) = (\xi\boldsymbol\varepsilon_1+\xi^{\prime}\widetilde{\boldsymbol\varepsilon}_1)G\cup\{\boldsymbol0\}$, where the elements $\xi,\xi^{\prime}$ of $\F$ are not both equal to zero.
Now set $\boldsymbol\lambda=\xi\boldsymbol\varepsilon_1+\xi^\prime\widetilde{\boldsymbol\varepsilon}_1+\widetilde{\boldsymbol\varepsilon}_2$ (resp., $\boldsymbol\lambda=\xi\boldsymbol\varepsilon_1+\xi^\prime\widetilde{\boldsymbol\varepsilon}_1+\boldsymbol\varepsilon_2$) if $\xi\ne0$ (resp., $\xi^\prime\ne0$).
Then $\boldsymbol\lambda\in \boldsymbol\mu G$.
Moreover, with $\widehat q=(q_i)$ where $q_1=0$ and $q_i=1$ for $i\ne1$ as in~\cite[Lemma~3.9]{IvanovaPallikaros2019} we see that $\xi\boldsymbol\varepsilon_1+\xi^\prime\widetilde{\boldsymbol\varepsilon}_1\in\overline{\boldsymbol\mu G}$.
Invoking~\cite[Lemma~3.1 and Remark~3.10(i)]{IvanovaPallikaros2019} we conclude that $\mathcal M^{\ast}_P\subseteq\overline{\boldsymbol\mu G}$.

\section{The $G$-submodule $\mathcal{M}^{\ast \ast }$}

\label{SecM**}

We assume that $|\mathbb{F}|>2$ throughout this section.

\subsection{Defining conditions}

The defining condition for the subset $\mathcal{M}^{\ast \ast }$ of $%
\boldsymbol{\Lambda }$ is that $\boldsymbol{\lambda }\in \mathcal{M}^{\ast
\ast }$ exactly when the algebra $\mathfrak{g}=\Theta ^{-1}(\boldsymbol{%
\lambda })$ has the property that $[v,v]\in\mathbb{F}$-sp$(v)$ for each $%
v\in V$. Clearly $\mathcal{M}^{\ast \ast }$ is a $G$-submodule of $%
\boldsymbol{\Lambda}$ containing $\mathcal{M}^{\ast}$. The defining property
for $\boldsymbol{\lambda}\in\boldsymbol{\Lambda}$ to belong to $\mathcal{M}%
^{\ast \ast }$ induces a function from $V-\{0\}$ to $\mathbb{F}$, where the
image $\omega_{\boldsymbol{\lambda}}(v)$ of a non-zero $v\in V$ is
determined by the relation $[v,v]=\omega_{\boldsymbol{\lambda}}(v)v$. By
assigning an arbitrary value for $\omega_{\boldsymbol{\lambda}}(0)$, this
last relation would then hold for all $v\in V$. Our aim is to extend $%
\omega_{\boldsymbol{\lambda}}$ to an element of $\widehat V$ so we define $%
\omega_{\boldsymbol{\lambda}}(0)=0$. We refer to $\omega_{\boldsymbol{\lambda%
}}$ as the \textbf{square factor function} for~$\mathfrak{g}$.

\medskip

We now check that $\omega_{\boldsymbol{\lambda}}$ is indeed a linear map
from $V$ to $\mathbb{F}$. For this, our assumption that $|\mathbb{F}|>2$ is
necessary. For simplicity, we will write $\omega$ in place of $\omega_{%
\boldsymbol{\lambda}}$ in the discussion that follows. First observe that $%
\omega (\alpha v)=\alpha \omega (v)$, for all $\alpha \in \mathbb{F}$ and $%
v\in V$. Expanding $[\alpha u+v,\alpha u+v]$ in two ways, we get%
\begin{eqnarray*}
\lbrack \alpha u+v,\alpha u+v] &=&\alpha ^{2}[u,u]+\alpha ([u,v]+[v,u])+[v,v]
\\
&=&\alpha ^{2}\omega (u)u+\alpha ([u,v]+[v,u])+\omega (v)v
\end{eqnarray*}%
and%
\begin{eqnarray*}
\lbrack \alpha u+v,\alpha u+v] &=&\omega (\alpha u+v)(\alpha u+v) \\
&=&\alpha \omega (\alpha u+v)u+\omega (\alpha u+v)v.
\end{eqnarray*}%
Therefore%
\begin{equation}
\alpha \omega (\alpha u+v)u+\omega (\alpha u+v)v=\alpha ^{2}\omega
(u)u+\alpha ([u,v]+[v,u])+\omega (v)v.  \label{square}
\end{equation}%
Taking $\alpha =1$ here gives%
\begin{equation*}
\omega (u+v)u+\omega (u+v)v=\omega (u)u+\omega (v)v+[u,v]+[v,u].
\end{equation*}%
Then%
\begin{equation}
\lbrack u,v]+[v,u]=(\omega (u+v)-\omega (u))u+(\omega (u+v)-\omega (v))v.
\label{commpre}
\end{equation}%
So%
\begin{eqnarray*}
\alpha \omega (\alpha u+v)u+\omega (\alpha u+v)v &=&\alpha ^{2}\omega
(u)u+\alpha ([u,v]+[v,u])+\omega (v)v \\
&=&\alpha ^{2}\omega (u)u+\alpha ((\omega (u+v)-\omega (u))u \\
&&+\alpha (\omega (u+v)-\omega (v))v+\omega (v)v.
\end{eqnarray*}%
Taking $u$ and $v$ to be linearly independent and equating coefficients of $%
u $ and of $v$ gives%
\begin{equation*}
\alpha \omega (\alpha u+v)=\alpha ^{2}\omega (u)+\alpha (\omega (u+v)-\omega
(u))
\end{equation*}%
and%
\begin{equation*}
\omega (\alpha u+v)=\alpha \omega (u+v)-\alpha \omega (v)+\omega (v).
\end{equation*}%
Cancelling an $\alpha $, $\alpha \neq 0$, in the first and equating the two
expressions for $\omega (\alpha u+v)$ shows that%
\begin{equation}
(\alpha -1)(\omega (u)+\omega (v)-\omega (u+v))=0.  \label{linear prelim}
\end{equation}

\bigskip

Since $\left\vert \mathbb{F}\right\vert >2$, we can take $\alpha $ and $%
\alpha -1$ both nonzero in (\ref{linear prelim}) and conclude that%
\begin{equation*}
\omega (u+v)=\omega (u)+\omega (v).
\end{equation*}

It is immediate that the last equation also holds when $u$ and $v$ are
linearly dependent in view of the fact that $\omega(\alpha
v)=\alpha\omega(v) $.

\medskip

Thus $\omega $ is a linear functional on $V$. Moreover, (\ref{commpre}) now
reads%
\begin{equation}
\lbrack u,v]+[v,u]=\omega (v)u+\omega (u)v.  \label{comm}
\end{equation}%
Since $\omega (v_{i})=\lambda _{iii}$ and $\omega (v_{j})=\lambda _{jjj}$,
we have%
\begin{equation*}
\lbrack v_{i},v_{j}]+[v_{j},v_{i}]=\lambda _{jjj}v_{i}+\lambda _{iii}v_{j}
\end{equation*}%
and we get (for distinct $i$, $j$ and $k$)%
\begin{eqnarray*}
\lambda _{iji}+\lambda _{jii} &=&\lambda _{jjj} \\
\lambda _{ijk}+\lambda _{jik} &=&0.
\end{eqnarray*}%
Recall also that $\lambda _{iij}=0$ for $i\neq j$ from the definition of $%
\mathcal{M}^{\ast \ast }$. So $\boldsymbol{\lambda }$ satisfies the
following conditions (all choices of subscripts are allowed but with
Convention ($\ddag $) observed):%
\begin{eqnarray}
\lambda _{ijk}+\lambda _{jik} &=&0  \notag \\
\lambda _{iij} &=&0  \label{M**} \\
\lambda _{iji}+\lambda _{jii} &=&\lambda _{jjj}.  \notag
\end{eqnarray}%
The conditions are independent, and there are%
\begin{equation*}
\binom{n}{2}(n-2)+n(n-1)+n(n-1)=\frac{n^{3}}{2}+\frac{n^{2}}{2}-n
\end{equation*}%
of them. Our aim is to show that the conditions~\eqref{M**} are in fact
defining conditions for $\mathcal{M}^{\ast \ast }$, so we suppose that these
conditions do hold for $\boldsymbol{\lambda }$. Then $[v_{i},v_{i}]=\lambda
_{iii}v_{i}$ and $[v_{i},v_{j}]+[v_{j},v_{i}]=\lambda _{jjj}v_{i}+\lambda
_{iii}v_{j}$. It follows that
\begin{eqnarray*}
\left[ \sum \xi _{i}v_{i},\sum \xi _{i}v_{i}\right] &=&\sum \xi
_{i}^{2}[v_{i},v_{i}]+\sum_{i\neq j}\xi _{i}\xi _{j}[v_{i},v_{j}] \\
&=&\sum \xi _{i}^{2}\lambda _{iii}v_{i}+\sum_{i<j}\xi _{i}\xi _{j}(\lambda
_{jjj}v_{i}+\lambda _{iii}v_{j}) \\
&=&\sum_{j}\left( \sum_{i}\xi _{i}\lambda _{iii}\right) \xi _{j}v_{j}.
\end{eqnarray*}%
With $v=\sum \xi _{i}v_{i}$, this says $[v,v]=\left( \sum_{i}\xi _{i}\lambda
_{iii}\right) v$. That shows that $\boldsymbol{\lambda }\in \mathcal{M}%
^{\ast \ast }$ and $\omega (\sum \xi _{i}v_{i})=\sum_{i}\xi _{i}\lambda
_{iii}$. Thus:

\begin{proposition}
\label{PropDefCondM**} Suppose $\left\vert \mathbb{F}\right\vert >2$. Then $%
\mathcal{M}^{\ast \ast }$ is defined by the conditions (\ref{M**}). %
Moreover, for $\boldsymbol{\lambda }\in \mathcal{M}^{\ast \ast }$, $\omega_{%
\boldsymbol{\lambda}}$ is a linear functional on $V$. Furthermore, $\dim
\mathcal{M}^{\ast \ast }=n^{3}/2-n^{2}/2+n$.
\end{proposition}

In the following remark we collect some applications of the various
relations on $\mathcal{M}^{\ast\ast}$ we have obtained so far in this
section.

\begin{remark}
\label{RemCM**} (i) Suppose that $\mathop{\rm char}\nolimits\mathbb{F}=2$
and that $\mathbb{F}\ne\mathbb{F}_2$. Let $\boldsymbol{\lambda}\in\mathcal{C}%
\cap\mathcal{M}^{\ast \ast }$. Considering the defining conditions~%
\eqref{M**} and the defining conditions for $\mathcal{C}$ and $\mathcal{K}$
(see Section~\ref{SubsecDefCondBases}) it is easy to deduce that $%
\boldsymbol{\lambda}\in\mathcal{K}$. The assumption on $\mathbb{F}$ clearly
ensures that $\mathcal{K}\subseteq\mathcal{C}\cap\mathcal{M}^{\ast \ast }$.
Hence $\mathcal{K}=\mathcal{C}\cap\mathcal{M}^{\ast \ast }$ in this case.

\smallskip (ii) Suppose now that $\mathop{\rm char}\nolimits\mathbb{F}\neq 2$%
. In $\mathfrak{g}=\Theta ^{-1}(\boldsymbol{\lambda })$, $[u,v]+[v,u]=\omega
(v)u+\omega (u)v$, by~(\ref{comm}). When $\boldsymbol{\lambda }\in \mathcal{C%
}$, this reads $[u,v]=\frac{1}{2}\omega (v)u+\frac{1}{2}\omega (u)v$. Thus $%
\boldsymbol{\lambda }\in \mathcal{M}_{(1,1)}^{\ast }$. In view of
Proposition~\ref{PropIntersecM*}, this implies that $\mathcal{C}\cap
\mathcal{M}^{\ast \ast }=\mathcal{C}\cap \mathcal{M}^{\ast }=\mathcal{M}%
_{(1,1)}^{\ast }$.%

\smallskip (iii) Invoking Proposition~\ref{PropIntersecM*}, it now follows
from item (ii) of this remark that in the case $\mathop{\rm char}\nolimits%
\mathbb{F}\ne2$, we have $\mathcal{N}\cap \mathcal{M}^{\ast\ast}=0$ (resp., $%
\mathcal{N}\cap \mathcal{M}^{\ast\ast}=\mathcal{M}^{\ast }_{(1,1)}$) if $%
\mathop{\rm char}\nolimits\mathbb{F}\nmid n+1$ (resp., $\mathop{\rm char}%
\nolimits\mathbb{F}\mid n+1$). However, if $\mathop{\rm char}\nolimits%
\mathbb{F}=2$, we have $\mathcal{N}\cap \mathcal{M}^{\ast\ast}=(\mathcal{T}%
\cap\mathcal{C})\cap \mathcal{M}^{\ast\ast} = \mathcal{T}\cap\mathcal{K}=
\mathcal{U}$, in view of item~(i) of this remark. Since $\dim\mathcal{N}=%
\frac{n^3}2+\frac{n^2}2-n$ and $\dim\mathcal{M}^{\ast\ast}=\frac{n^3}2-\frac{%
n^2}2+n$, we get that $\dim(\mathcal{N}+\mathcal{M}^{\ast\ast})=n^3-\dim%
\mathcal{U}=\frac{n^3}2+\frac{n^2}2+n$, when $\mathop{\rm char}\nolimits%
\mathbb{F}=2$.
\end{remark}

Now for any algebra $\mathfrak{g}=\Theta^{-1}(\boldsymbol{\lambda})$ with $%
\boldsymbol{\lambda}\in\mathcal{M}^{\ast\ast}$, writing $\omega(\boldsymbol{%
\lambda},v)=\omega_{\boldsymbol{\lambda}}(v)$, we have $[u,v]+[v,u]=\omega(%
\boldsymbol{\lambda},v)u+\omega(\boldsymbol{\lambda},u)v$ from~\eqref{comm}.
For a linear functional $\mu$ on $V$, the trace of $u\mapsto\mu(v)u$ (a
diagonal map) is $n\mu(v)$, and the trace of $u\mapsto\mu(u)v$ is $\mu(v)$.
Thus $\mathrm{tr}(\boldsymbol{\lambda},v)+\widetilde{\mathrm{tr}}(%
\boldsymbol{\lambda},v)=(n+1)\omega(\boldsymbol{\lambda},v)$. So $\mathrm{tr}%
=-\widetilde{\mathrm{tr}}$ on $\mathcal{K}$, and this equality will hold on $%
\mathcal{M}^{\ast\ast}$ itself exactly when $\mathop{\rm char}\nolimits%
\mathbb{F}$ divides $n+1$. In that case, $\mathcal{T}\cap \mathcal{M}%
^{\ast\ast}=\widetilde{\mathcal{T}}\cap \mathcal{M}^{\ast\ast}$. We prove
the converse

\begin{proposition}
\label{PropTIntersM**} We have $\mathcal{T}\cap \mathcal{M}^{\ast\ast}=%
\widetilde{\mathcal{T}}\cap \mathcal{M}^{\ast\ast}$ if, and only if, $%
\mathop{\rm char}\nolimits\mathbb{F}$ divides $n+1$.
\end{proposition}

\begin{proof}
Suppose $\mathcal T\cap \mathcal{M}^{\ast\ast}=\widetilde{\mathcal T}\cap \mathcal{M}^{\ast\ast}$.
Also let $\boldsymbol\lambda=(\lambda_{ijk})\in\boldsymbol\Lambda$, where $\lambda_{111}=1$, $\lambda_{212}=-1$, $\lambda_{122}=2$, $\lambda_{1jj}=1$ for all $j>2$, and all other $\lambda_{ijk}$ are equal to zero.
It is then immediate from equation~\eqref{trlabdatilde} and conditions~\eqref{M**} that $\boldsymbol\lambda\in\widetilde{\mathcal T}\cap \mathcal{M}^{\ast\ast}$.
Since $\mathcal T\cap \mathcal{M}^{\ast\ast}=\widetilde{\mathcal T}\cap \mathcal{M}^{\ast\ast}$, we have $\boldsymbol\lambda\in\mathcal T$ also, so $\sum_j\lambda_{1jj}=0$ from equation~\eqref{basisN}.
But $\sum_j\lambda_{1jj}=1+2+(n-2)=n+1$.
We conclude that $n+1=0$ in $\F$.
\end{proof}

Now let $\boldsymbol{\mu}=(\mu_{ijk})\in\boldsymbol{\Lambda}$ where $%
\mu_{111}=1$, $\mu_{j1j}=1$ for all $j>1$ and all other $\mu_{ijk}$ are
equal to 0. It follows from~\eqref{trbydual} and~\eqref{M**} that $%
\boldsymbol{\mu}\in\mathcal{M}^{\ast\ast}-\mathcal{T}$. Since $\boldsymbol{%
\Lambda}/\mathcal{T}\,(\simeq\widehat V)$ is irreducible, we can deduce that
$\mathcal{T}+\mathcal{M}^{\ast\ast}=\boldsymbol{\Lambda}$. Hence, involving
Propositions~\ref{PropL/TIsomToHatV} and~\ref{PropDefCondM**} we get $n^3 =
\dim(\mathcal{T }+ \mathcal{M}^{\ast\ast}) = (n^3-n)+(n^3/2-n^2/2+n)-\dim(%
\mathcal{T}\cap \mathcal{M}^{\ast})$. It follows that $\dim(\mathcal{T}\cap
\mathcal{M}^{\ast\ast})=n^3/2-n^2/2=\dim\mathcal{U}+n$. We thus have:

\begin{corollary}
\label{CorTTM**} Suppose that $\mathop{\rm char}\nolimits\mathbb{F}\mid n+1$%
. Then $(\mathcal{T}\cap\widetilde{\mathcal{T}})\cap\mathcal{M}^{\ast\ast}=%
\mathcal{T}\cap \mathcal{M}^{\ast\ast}$. In particular, $\dim((\mathcal{T}%
\cap\widetilde{\mathcal{T}})\cap\mathcal{M}^{\ast\ast})=n^3/2-n^2/2=\dim%
\mathcal{U}+n$.
\end{corollary}

\begin{proof}
Invoking Proposition~\ref{PropTIntersM**} we get that $(\mathcal T\cap\widetilde{\mathcal T})\cap\mathcal M^{\ast\ast}= \mathcal T\cap(\widetilde{\mathcal T}\cap\mathcal M^{\ast\ast})= \mathcal T\cap(\mathcal T\cap\mathcal M^{\ast\ast})=\mathcal T\cap\mathcal M^{\ast\ast}$, whenever $\Char\F$ divides $n+1$.
\end{proof}

Proposition~\ref{PropTIntersM**} and Corollary~\ref{CorTTM**} will play some
part in Section~\ref{SectGLVstructure}.

\subsection{The action of $GL(V)$}

Suppose that $\boldsymbol{\lambda }\in \mathcal{M}^{\ast \ast }$, with $%
[v,v]=\omega (v)v$ in the algebra $\mathfrak{g}=\Theta ^{-1}(\boldsymbol{%
\lambda })$. Then, comparing with Definition~\ref{DefAction}, we have for $%
\mathfrak{g}g$,
\begin{equation*}
\omega ^{\prime }(v)v=[v,v]^{\prime }=g^{-1}[gv,gv]=g^{-1}\omega
(gv)(gv)=\omega (gv)v.
\end{equation*}%
Thus $\omega ^{\prime }(v)=\omega (gv)=(\omega g)(v)$ (by the definition for
right action). Tagging $\omega $ for $\boldsymbol{\lambda }$ as $\omega _{%
\boldsymbol{\lambda }}$, we also have that $\omega _{(\boldsymbol{\lambda }+%
\boldsymbol{\mu })}=\omega _{\boldsymbol{\lambda }}+\omega _{\boldsymbol{\mu
}}$ and $\omega _{\alpha \boldsymbol{\lambda }}=\alpha \omega _{\boldsymbol{%
\lambda }}$. So $\boldsymbol{\lambda }\mapsto \omega _{\boldsymbol{\lambda }%
} $ is a $G$-homomorphism from $\mathcal{M}^{\ast \ast }$ to $\widehat V$.

Next we show that this $G$-homomorphism is surjective. For this, let an
arbitrary $\mu\in\widehat V$ be given with $\mu(v_i)=\mu_i$ for $1\le i\le n$%
. Define $\boldsymbol{\lambda}$ by $\lambda_{iii}=\mu_i$ and $%
\lambda_{iji}=\mu_j$ (for $i\ne j$) and all other components to be zero.
Clearly $\boldsymbol{\lambda}\in\mathcal{M}^{\ast \ast}$ since the defining
conditions~\eqref{M**} are all satisfied. Moreover, in $\Theta^{-1}(%
\boldsymbol{\lambda})$ we have that $[v_i,v_j]+[v_j,v_i]=\mu_jv_i+\mu_iv_j$,
true for all $i,j$ (including $i=j$). Now let $v=\sum_i\xi_iv_i\in V$.
Comparing with the discussion immediately before Proposition~\ref%
{PropDefCondM**} we get that $[v,v]=(\sum_i\xi_i\mu_i)v=\mu(v)v$. It follows
that $\mu(v)=\omega_{\boldsymbol{\lambda}}(v)$ for all $v\in V$. We thus
have:

\begin{corollary}
\label{CorolM**K} Suppose $|\mathbb{F}|>2$. The map $\boldsymbol{\lambda }%
\mapsto \omega _{\boldsymbol{\lambda }}$ is a $G$-homomorphism from $%
\mathcal{M}^{\ast \ast }$ onto $\widehat{V}$, the dual space of $V$ as a
right $G$-module. The kernel is $\mathcal{K}$. In particular, $\mathcal{M}%
^{\ast \ast }/\mathcal{K}\backsimeq \widehat{V}$ as $G$-modules.
\end{corollary}

\section{Linear degeneration}

\label{SectLinDegen}

Degeneration can be used for proving that certain $G$-submodules of the
space $\boldsymbol{\Lambda }$ of algebra structures over a field $\mathbb{F}$
are irreducible. Recall that for structure vectors $\boldsymbol{\lambda}$
and $\boldsymbol{\lambda}^{\prime}$, we say that $\boldsymbol{\lambda}$
degenerates to $\boldsymbol{\lambda}^{\prime}$ (denoted by $\boldsymbol{%
\lambda }\rightarrow \boldsymbol{\lambda^{\prime}}$) if $\boldsymbol{\lambda}%
^{\prime}$ belongs to Zariski-closure of the $G$-orbit of $\boldsymbol{%
\lambda}$. As an example, consider the submodule $\mathcal{U}$.
Recall Remark~\ref{RemUeta} that $\mathcal{U}=\boldsymbol{\eta }(\mathbb{F}%
G) $, where $\boldsymbol{\eta }=\mathbf{123}-\mathbf{213} $. Moreover, if $%
\boldsymbol{\lambda }\in \mathcal{M}^{\ast \ast }$ but $\boldsymbol{\lambda }%
\notin \mathcal{M}^{\ast }$, then $\boldsymbol{\lambda }$ degenerates to $%
\boldsymbol{\eta }$ by \cite[Lemma 4.4]{IvanovaPallikaros2019} applied to
structure vectors, when $\mathbb{F}$ is infinite. Since $\boldsymbol{\lambda
}(\mathbb{F}G)$ is closed, $\boldsymbol{\eta }\in \boldsymbol{\lambda }(%
\mathbb{F}G)$. Then $\mathcal{U}\mathcal{=\boldsymbol{\eta }(}\mathbb{F}%
G)\subseteq \boldsymbol{\lambda }(\mathbb{F}G)$. So if $\boldsymbol{\lambda }%
\in \mathcal{U}{-}\mathcal{M}^{\ast }$, then $\boldsymbol{\lambda }(\mathbb{F%
}G)=\mathcal{U}$. In particular, $\mathcal{U}/\mathcal{U}{\cap }\mathcal{M}%
^{\ast }$ is irreducible.

As we pointed out, when $\mathbb{F}$ is finite, \emph{everything} is closed.
As a substitute for closed sets we use $G$-submodules instead, and we make
an apparently toothless definition:

\begin{definition}
\emph{Let }$\boldsymbol{\lambda }$ and $\boldsymbol{\lambda }^{\prime }$ be
structure vectors over an arbitrary field $\mathbb{F}$.\emph{\ Then }$%
\boldsymbol{\lambda }^{\prime }$\emph{\ is called a \textbf{linear
degeneration} of }$\boldsymbol{\lambda }\emph{\ }$if $\boldsymbol{\lambda }%
^{\prime }\in \boldsymbol{\lambda }(\mathbb{F}G)$\emph{.}
\end{definition}

\noindent We also say that $\boldsymbol{\lambda }$ \textbf{linearly
degenerates }to $\boldsymbol{\lambda }^{\prime }$ and write $\boldsymbol{%
\lambda }\looparrowright \boldsymbol{\lambda }^{\prime }$. (Clearly if $%
\mathbb{F}$ is infinite and $\boldsymbol{\lambda }\rightarrow \boldsymbol{%
\lambda^{\prime}}$, then $\boldsymbol{\lambda }\looparrowright \boldsymbol{%
\lambda }^{\prime }$.) What makes this actually useful is that there is
something of an analogue of \cite[Lemma 3.9]{IvanovaPallikaros2019}. As in
that lemma, let $\widehat{q}$ be a sequence $(q_{1},\ldots ,q_{n})$ of
integers, and for $\boldsymbol{\lambda }\in \boldsymbol{\Lambda }$, define $%
\boldsymbol{\lambda }(\widehat{q})$ by ${\lambda }(\widehat{q}%
)_{ijk}=\lambda _{ijk}$ if $q_{i}+q_{j}-q_{k}=0$ and 0 if not.

\begin{theorem}
\label{TheorLinDeg} \label{lindeg}Let $\boldsymbol{\lambda }\in \boldsymbol{%
\Lambda }$ and suppose that $\lambda _{ijk}=0$ whenever $q_{i}+q_{j}-q_{k}<0$%
. Then if $\max (q_{i}+q_{j}-q_{k})<\left\vert \mathbb{F}\right\vert -1$, $%
\boldsymbol{\lambda }(\widehat{q})$ is a linear degeneration of $\boldsymbol{%
\lambda }$.
\end{theorem}

\begin{proof}
Let $\tau\in\F-\{0\}$ and take $g(\tau )\in G$ so that $[g(\tau )]$ is  the diagonal matrix having $\tau^{q_{i}}$ as its $(i,i)$-entry.
Then let $\boldsymbol{\lambda }(\tau )=\boldsymbol{\lambda }g(\tau )$, so
that $\lambda (\tau )_{ijk}=\tau ^{q_{i}+q_{j}-q_{k}}\lambda _{ijk}$, as in
\cite[Lemma 3.9]{IvanovaPallikaros2019}.
Suppose that $\zeta $ is a linear functional on $\boldsymbol{\Lambda }$
with $\boldsymbol{\lambda }(\mathbb{F}G)$ in its kernel. Then $\zeta (\boldsymbol{%
\lambda }(\tau ))=0$.
If $\zeta (\mathbf{ijk})=\zeta _{ijk}$, then $\zeta (%
\boldsymbol{\lambda }(\tau ))=\sum_{i,j,k}\zeta _{ijk}\tau
^{q_{i}+q_{j}-q_{k}}\lambda _{ijk}$.
Now let the polynomial $f(x)\in\F[x]$ be defined by $f(x)=\sum\zeta _{ijk}\lambda_{ijk}x^{q_{i}+q_{j}-q_{k}}$, where the sum is taken over all $(i,j,k)$ with $\lambda_{ijk}\ne0$ and, as usual, $x^0$ denotes the constant term $1$.
Then $f(\tau)=\zeta(\boldsymbol{\lambda }(\tau))\,(=0)$ for $\tau\ne0$ and $f(0)=\zeta(\boldsymbol{\lambda }(\hat q))$.
As the degree of $f(x)$ is strictly less than $\left\vert \mathbb{F}\right\vert-1$ it must be the zero polynomial.
So $\zeta (\boldsymbol{\lambda }(\widehat{q}))=0$.
This being the case for all linear functionals $\zeta $ having the
subspace $\boldsymbol{\lambda }(\mathbb{F}G)$ in their kernels, we get $\boldsymbol{%
\lambda }(\widehat{q})\in \boldsymbol{\lambda }(\mathbb{F}G)$. That is, $\boldsymbol{%
\lambda }(\widehat{q})$ is a linear degeneration of $\boldsymbol{\lambda }$.
\end{proof}

In the following example we discuss some applications of Theorem~\ref%
{TheorLinDeg}.

\begin{example}
\label{ExampleApplThLD} (i) The sequence $\widehat{q}$ used in \cite[Lemma
4.4]{IvanovaPallikaros2019} had just 1's and 2's in it, making $\max
(q_{i}+q_{j}-q_{k})=3$ (this maximum in general is $2\max (q_{1},\ldots
,q_{n})-\min (q_{1},\ldots ,q_{n})$). So if $\boldsymbol{\lambda }\in
\mathcal{M}^{\ast \ast }$ and $\boldsymbol{\lambda }\notin \mathcal{M}^{\ast
}$, then $\boldsymbol{\lambda }$ linearly degenerates to $\boldsymbol{\eta }$
if $\left\vert \mathbb{F}\right\vert \geq 5$. With this restriction, $%
\mathcal{U}/\mathcal{U}\cap \mathcal{M}^{\ast }$ is still irreducible. We
save examining smaller fields until later.

\smallskip (ii) Similarly, now invoking~\cite[Lemma~5.4]%
{IvanovaPallikaros2019} we get that if $\boldsymbol{\lambda}\in\boldsymbol{%
\Lambda}-\mathcal{M}^{\ast\ast}$ and $|\mathbb{F}|\ge5$, then $\boldsymbol{%
\lambda}$ linearly degenerates to $\boldsymbol{\delta}$. Note that in the
proof of that lemma, the $q_{i}$ are either 1 or 2, thus $%
\max(q_{i}+q_{j}-q_{k})=3$ again. %
\end{example}

However, here is a linear degeneration important for the structure of $%
\boldsymbol{\Lambda }$ when $\mathrm{char}\mathbb{F}=2$:

\begin{proposition}
\label{C/K}If $\mathrm{char}\mathbb{F}=2$ and $\left\vert \mathbb{F}%
\right\vert \geq 8$, then the $G$-module $\mathcal{C}/\mathcal{K}$ is
irreducible.
\end{proposition}

\begin{proof}
Recall from Remark~\ref{RemCM**}(i) that  $\mathcal{K}=\mathcal{C}\cap \mathcal  M^{\ast \ast }$ if $\Char\F=2$ and $\mathbb{F}%
\neq \mathbb{F}_{2}$.
Hence, Example~\ref{ExampleApplThLD}(ii) applies when $\boldsymbol{\lambda }%
\in \mathcal{C}-\mathcal{K}$ to show that $\boldsymbol{%
\lambda }\looparrowright \mathbf{112}=\boldsymbol{\delta }$ when $%
\left\vert \mathbb{F}\right\vert \geq 5$.
So for any $\mathbb{F}$ satisfying the hypothesis, infinite or not, $\mathbf{%
112}\in \boldsymbol{\lambda }(\mathbb{F}G)$. Then by index permutations, we get $%
\mathbf{iij}\in \boldsymbol{\lambda }(\mathbb{F}G)$ for all $i$ and $j\neq i$.
Our goal is to prove that $\boldsymbol{\lambda }(\mathbb{F}G)+\mathcal{K}=%
\mathcal{C}$. Recalling Remark~\ref{RemSkewAlg}(ii) that the cosets $\mathbf{iii}+%
\mathcal{K}$ and $\mathbf{iij}+\mathcal{K}$ (for $i\ne j$) form a basis for $%
\mathcal{C}/\mathcal{K}$, we need the triples $\mathbf{iii}$ to be
in $\boldsymbol{\lambda }(\mathbb{F}G)+\mathcal{K}$.\newline
\qquad Let $\alpha \in \mathbb{F}$, $\alpha \neq 0$, and let $g\in G$ with
\begin{equation*}
[g]=%
\begin{bmatrix}
1 & \alpha & 0 \\
0 & 1 & 0 \\
0 & 0 & I_{n-2}%
\end{bmatrix}%
.
\end{equation*}%
Then by Equation~\eqref{TripleChange}%
\begin{eqnarray*}
\mathbf{112}g &=&\sum_{i,j,k}g_{1i}g_{1j}g_{k2}^{(-1)}\mathbf{ijk} \\
&=&\mathbf{112}-\alpha \mathbf{111}+\alpha ^{2}\mathbf{222}-\alpha ^{3}%
\mathbf{221} \\
&&+\alpha (\mathbf{122}+\mathbf{212})-\alpha ^{2}(\mathbf{121}+\mathbf{211})
\end{eqnarray*}%
The last two terms are in $\mathcal{K}$, and the first and fourth are in
$\boldsymbol{\lambda }(\mathbb{F}G)$. So $-\alpha \mathbf{111}+\alpha ^{2}%
\mathbf{222}+\mathcal{K}\in \boldsymbol{\lambda }(\mathbb{F}G)+\mathcal{K}%
$. As this holds for any $\alpha $, we get that $\mathbf{111}$ and $%
\mathbf{222}$ separately belong to $\boldsymbol{\lambda }(\mathbb{F}G)+\mathcal{K%
}$, and now permutations show that all $\mathbf{iii}$ are in $\boldsymbol{%
\lambda }(\mathbb{F}G)+\mathcal{K}$, as needed. Incidentally, the
equality $\mathbf{112}(\mathbb{F}G)+\mathcal{K}=\mathcal{C}$ will
hold for $\mathbb{F}=\mathbb{F}_{4}$, too, there still being enough $\alpha $%
's for this last argument to work.
\end{proof}

In the next section we shall see that $\mathcal{C}/\mathcal{K}$ is also
irreducible when $\mathbb{F}=\mathbb{F}_{4}$.

\section{Characteristic 2}

\label{SectChar2}

For this section assume that the scalar field has characteristic 2 and is
perfect, so that the Frobenius map $\alpha \mapsto \alpha ^{2}$ is an
automorphism. The goal here is to analyze the quotient $\mathcal{C}/\mathcal{%
K}$ as a $G$-module. We have seen that, in fact, linear degeneration implies
that the module is irreducible for $\left\vert \mathbb{F}\right\vert \geq 8$
(Proposition \ref{C/K}). But we want to point out some other features of
that module.

\subsection{An action on $\Gamma V$}

Let $\boldsymbol{\lambda }\in \mathcal{C}$ and $\mathfrak{g}=\Theta ^{-1}(%
\boldsymbol{\lambda })$. For $v\in V$, define $\Sigma _{\boldsymbol{\lambda }%
}(v)=[v,v]$, the squaring map. Since $\mathfrak{g}$ is commutative, $\Sigma
_{\boldsymbol{\lambda }}$ is additive; but as $[\alpha v,\alpha v]=\alpha
^{2}[v,v]$, $\Sigma _{\boldsymbol{\lambda }}$ is semilinear with respect to
the Frobenius $\alpha \mapsto \alpha ^{2}$. The set $\Gamma V$ of semilinear
maps $V\longrightarrow V$ relative to the Frobenius is an $\mathbb{F}$-space
(and as such, isomorphic to $\mathbb{F}^{n\times n}$), and the map $\Sigma :%
\boldsymbol{\lambda }\mapsto \Sigma _{\boldsymbol{\lambda }} $ is linear.
The kernel of $\Sigma $ is $\mathcal{K}$, so the space $\mathcal{C}/\mathcal{%
K}$ is isomorphic to a subspace of $\Gamma V$. What about the $G$-action?
Let $\boldsymbol{\lambda }^{\prime }=\boldsymbol{\lambda }g$ and $\mathfrak{g%
}^{\prime }=\Theta ^{-1}(\boldsymbol{\lambda }^{\prime })$. Then $%
[v,v]^{\prime }=g^{-1}[gv,gv]$ (see Definition~\ref{DefAction}), so that $%
\Sigma _{\boldsymbol{\lambda }^{\prime }}(v)=g^{-1}\Sigma _{\boldsymbol{%
\lambda }}(gv)$. Since we want $\boldsymbol{\lambda }\mapsto \Sigma _{%
\boldsymbol{\lambda }}$ to be a $G$-map of \emph{right} $G$-modules, the
required action on $\Gamma V$ is defined by $\varphi \ast g:v\mapsto
g^{-1}\varphi (gv)$; that is, $\varphi \ast g=g^{-1}\circ \varphi \circ g$.

Here is this last formula in matrix terms: recall from Section~\ref%
{SectAlgSetUp} that the standard basis for $V$ is $v_{1},\ldots ,v_{n}$, and
a linear transformation $g$ (acting on the left) is presented as matrix $%
[g_{ij}]$ with $gv_{j}=\sum_{i}g_{ij}v_{i}$. The entries of transformation $%
\varphi $ as a matrix are given by $\varphi (v_{j})=\sum_{i}\varphi
_{ij}v_{i}$ too, but in the composition $\varphi \circ g$, we have
\begin{eqnarray*}
\varphi (g(v_{j})) &=&\varphi
(\sum_{i}g_{ij}v_{i})=\sum_{i}g_{ij}^{2}\varphi (v_{i}) \\
&=&\sum_{i,k}g_{ij}^{2}\varphi _{ki}v_{k}=\sum_{k}\left( \sum_{i}\varphi
_{ki}g_{ij}^{2}\right) v_{k}.
\end{eqnarray*}%
This makes the final matrix for $\varphi \ast g=g^{-1}\circ \varphi \circ g$
to be $[\varphi _{ij}^{\prime }]$, with $\varphi _{ij}^{\prime
}=\sum_{k,l}g_{ik}^{(-1)}\varphi _{kl}g_{lj}^{2}$.

\medskip

In order to simplify notation, it will be convenient from this point and up
to the end of Subsection~\ref{SubsecG-modStrGammaV} to regard both $\mathrm{%
End}_{\mathbb{F}}V$ and $\Gamma V$ as spaces of matrices (and accordingly
for the subset $GL(V)$ of $\mathrm{End}_{\mathbb{F}}V$). More specifically,
by $h\in\mathrm{End}_{\mathbb{F}}V$ (resp., by $\psi\in\Gamma V$) we will
mean the matrix representing a suitable transformation with respect to the
standard basis of $V$. (Also note that the action on $\boldsymbol{\Lambda}$
(resp., $\Gamma V$) by the group of invertible linear transformations from $%
V $ to $V$ we have been considering, induces in an obvious way an action on $%
\boldsymbol{\Lambda}$ (resp., $\Gamma V$) by the corresponding group of
matrices.)

Thus, with the notation described immediately above,
we have (for $\varphi\in\Gamma V$ and $g\in GL(V)$)
\begin{equation}
\varphi \ast g=g^{-1}\varphi g^{(2)},  \label{starop}
\end{equation}%
where the very last matrix is $g$ with its entries squared.

Now we corroborate (\ref{starop}) by using the map $\Sigma $. To match the
matrix indexing, we present the relative basis members for $\mathcal{C}/%
\mathcal{K}$ as $\mathbf{jji}$, allowing $i=j$. Then $\Sigma _{\mathbf{jji}%
}\,(v_{j})=v_{i}$, and all the other basis products with $\Sigma _{\mathbf{%
jji}}$ are 0. So $\Sigma _{\mathbf{jji}}=e_{ij}$, the $ij$ matrix unit. For $%
g\in G$, we have%
\begin{eqnarray*}
\mathbf{jji}g &=&\sum_{a,b,c}g_{ja}g_{jb}g_{ci}^{(-1)}\mathbf{abc} \\
&=&\sum_{a,c}g_{ja}g_{ja}g_{ci}^{(-1)}\mathbf{aac}%
+\sum_{a<b,c}g_{ja}g_{jb}g_{ci}^{(-1)}(\mathbf{abc}+\mathbf{bac})
\end{eqnarray*}%
by \eqref{TripleChange}. Applying the (linear) map $\Sigma $ and observing
that $\mathbf{abc}+\mathbf{bac}\in \mathcal{K}$, we get%
\begin{equation*}
\Sigma _{\mathbf{jji}g}=\Sigma _{\mathbf{jji}}\ast
g=\sum_{a,c}g_{ci}^{(-1)}g_{ja}^{2}\Sigma _{\mathbf{aac}}.
\end{equation*}%
Since $\Sigma _{\mathbf{aac}}=e_{ca}$, the last equation becomes
\begin{equation*}
e_{ij}\ast g=\sum_{a,c}g_{ci}^{(-1)}g_{ja}^{2}e_{ca}.
\end{equation*}%
Then with $\varphi =\sum_{i,j}\varphi _{ij}e_{ij}$, this gives%
\begin{equation*}
\varphi \ast g=\sum_{a,c}\sum_{i,j}g_{ci}^{(-1)}\varphi
_{ij}g_{ja}^{2}e_{ca}.
\end{equation*}%
So $\varphi \ast g=g^{-1}\varphi g^{(2)}$ indeed, the formula in (\ref%
{starop}).

\subsection{$G$-module structure of $\Gamma V$}

\label{SubsecG-modStrGammaV}

Continuing to regard $\Gamma V$ and $\mathrm{End}_{\mathbb{F}}V$ as spaces
of matrices, recall that in Proposition \ref{C/K} we saw that $\Gamma V$ is
irreducible as a $G$-module under the action $\varphi \mapsto \varphi \ast g$
when $\left\vert \mathbb{F}\right\vert \geq 8$. First notice that if the
entries in $g$ are actually all in $\mathbb{F}_{2}$, then $g^{(2)}=g$ and $%
\varphi \ast g=g^{-1}\varphi g$. In particular, permutation matrices still
give permutations with the $\ast $ operation. Thus if $W$ is a $G$-submodule
of $\Gamma V$ and one matrix unit $e_{ij}\in W$ with $i\neq j$, then all
such $e_{ij}$ are in $W$. Similarly, one $e_{ii}$ in $W$ implies that all $%
e_{ii}\in W$. Our identification of $\Sigma _{\mathbf{jji}}$ with $e_{ij}$,
along with these comments, is in effect what is involved in the proof of
Proposition \ref{C/K}.

Let $W$ be a $G$-submodule of $\Gamma V$. In presenting matrices we shall
often write them in terms of the matrix units. Let $e$ and $f$ be two
\textquotedblleft off-diagonal\textquotedblright\ ($i\neq j$) matrix units
for which $ef=fe=0$. Then consider the map%
\begin{equation*}
e\&f:\varphi \mapsto \varphi \ast (I+e+f)+\varphi \ast (I+e)+\varphi \ast
(I+f)+\varphi
\end{equation*}%
for $\varphi \in \Gamma V$. Because the entries in the three matrices are
all in $\mathbb{F}_{2}$ and each matrix has order 2, the sum here is%
\begin{equation*}
(I+e+f)\varphi (I+e+f)+(I+e)\varphi (I+e)+(I+f)\varphi (I+f)+\varphi ,
\end{equation*}%
and this simplifies to $e\varphi f+f\varphi e$. If $\varphi \in W$, then $%
e\&f(\varphi )\in W$. Suppose that $\varphi $ has a nonzero off-diagonal
entry. As $W$ is closed under permutations, we may assume that $\varphi
_{12}\neq 0$. Then $e_{21}\&e_{31}(\varphi )=\varphi _{12}e_{31}+\varphi
_{13}e_{21}$, and $e_{13}\&e_{23}(\varphi _{12}e_{31}+\varphi
_{13}e_{21})=\varphi _{12}e_{23}$. So $e_{23}\in W$, and then $W$ contains
all $e_{ij}$, $i\neq j$. Now take $g=I+\alpha e_{21}$, $\alpha \neq 0$. Then
$g^{-1}=g$ and $g^{(2)}=I+\alpha ^{2}e_{21}$. We get%
\begin{equation}
e_{12}\ast g+e_{12}=\alpha e_{22}+\alpha ^{2}e_{11}+\alpha ^{3}e_{21}.
\label{eii}
\end{equation}%
Since $e_{21}\in W$, we conclude that $e_{22}+\alpha e_{11}\in W$. If $%
\left\vert \mathbb{F}\right\vert \geq 4$, we then get $e_{11}$ (and $e_{22}$%
) in $W$, and then by permutations, all $e_{ii}\in W$. So $W=\Gamma V$.

On the other hand, if all we have to begin with is that $e_{11}\in W$, take $%
g=I+e_{12}$ to produce%
\begin{equation*}
(I+e_{12})e_{11}(I+e_{12})+e_{11}=e_{12}
\end{equation*}%
in $W$, and then apply the preceding discussion to see that $W=\Gamma V$
again. In summary:

\begin{proposition}
\label{PropGammaV} If $\left\vert \mathbb{F}\right\vert \geq 4$, then $%
\mathcal{C}/\mathcal{K}\backsimeq \Gamma V$ is an irreducible $G$-module.
\end{proposition}

\section{Transvection degenerations}

\label{SectTransvec}

In this section we use transvections to examine linear degenerations. We
shall work with both the algebras and their structure vectors. A typical
algebra is $\mathfrak{g}$, with structure vector $\boldsymbol{\lambda }%
=\Theta (\mathfrak{g})$. Let $g$ be the transvection $g:v\mapsto v+\zeta
(v)z $, where $\zeta $ is a nonzero linear functional on $V$ with $\zeta
(z)=0$. Then let $\mathfrak{g}_{1}=\mathfrak{g}g$, so that the product in $%
\mathfrak{g}_{1}$ is given by%
\begin{eqnarray*}
\lbrack u,v]_{1} &=&g^{-1}[gu,gv] \\
&=&g^{-1}[u+\zeta (u)z,v+\zeta (v)z] \\
&=&g^{-1}([u,v]+\zeta (u)[z,v]+\zeta (v)[u,z]+\zeta (u)\zeta (v)[z,z]) \\
&=&[u,v]+\zeta (u)[z,v]+\zeta (v)[u,z]+\zeta (u)\zeta (v)[z,z] \\
&&-(\zeta ([u,v])+\zeta (u)\zeta ([z,v])+\zeta (v)\zeta ([u,z])+\zeta
(u)\zeta (v)\zeta ([z,z]))z.
\end{eqnarray*}%
(We shall use this kind of indexing in what follows.) Then $\Theta(\mathfrak{%
g}_{2})$, where 
$\mathfrak{g}_{2}=\mathfrak{g}_{1}-\mathfrak{g}$, is in $\boldsymbol{\lambda
}(\mathbb{F}G\mathbb{)}$. If we apply the same computation using $\alpha
\zeta $ in place of $\zeta $ ($\alpha \neq 0)$ to get $\mathfrak{g}_{3}$,
and then take $\mathfrak{g}_{4}=\mathfrak{g}_{3}-\alpha \mathfrak{g}_{2}$,
we end up with%
\begin{eqnarray*}
\lbrack u,v]_{4} &=&(\alpha ^{2}-\alpha )\zeta (u)\zeta (v)[z,z]-(\alpha
^{2}-\alpha )\zeta (u)\zeta ([z,v])z \\
&&-(\alpha ^{2}-\alpha )\zeta (v)\zeta (u,z]))z-(\alpha ^{3}-\alpha )\zeta
(u)\zeta (v)\zeta ([z,z])z.
\end{eqnarray*}%
Assuming that $\left\vert \mathbb{F}\right\vert >2$, we can take $\alpha
\neq 1$, scale by dividing by $-(\alpha ^{2}-\alpha )$, and conclude that $%
\Theta(\mathfrak{g}_{5})$, where $\mathfrak{g}_{5}$ has product%
\begin{eqnarray}
\lbrack u,v]_{5} &=&\zeta (u)\zeta ([z,v])z+\zeta (v)\zeta ([u,z]))z  \notag
\\
&&-\zeta (u)\zeta (v)[z,z]-(\alpha +1)\zeta (u)\zeta (v)\zeta ([z,z])z
\label{bracket5}
\end{eqnarray}%
is in $\boldsymbol{\lambda }(\mathbb{F}G\mathbb{)}$. That is, $\boldsymbol{%
\lambda }$ linearly degenerates to $\Theta(\mathfrak{g}_{5})$. We shall use
this degeneration in two cases. Therefore, for the rest of this section we
assume that $|\mathbb{F}|>2$.

\subsection{$\boldsymbol{\protect\lambda }\in \mathcal{M}^{\ast\ast}$}

Let $\boldsymbol{\lambda }\in \mathcal{M}^{\ast \ast }$, and let $\mathfrak{g%
}=\Theta ^{-1}(\boldsymbol{\lambda })$, with product $[\,,\,]$. Then, as
above, we obtain an algebra $\mathfrak{g}_{5}$ with $\Theta (\mathfrak{g}%
_{5})\in \boldsymbol{\lambda }\mathbb{F}G$ whose product is%
\begin{eqnarray*}
\lbrack u,v]_{5} &=&\zeta (u)\zeta ([z,v])z+\zeta (v)\zeta ([u,z])z-\zeta
(u)\zeta (v)[z,z] \\
&&-(\alpha +1)\zeta (u)\zeta (v)\zeta ([z,z])z.
\end{eqnarray*}%
Here $\zeta $ is a nonzero linear functional on $V$ which is $0$ on the
chosen vector $z\neq 0$. (The transvection used for the linear degeneration
is $v\mapsto v+\zeta (v)z$.) The nonzero scalar $\alpha $ is also not $1$.
Now assume that $\boldsymbol{\lambda }\notin \mathcal{M}^{\ast }$. Then
there are two vectors $a$ and $b$ for which $a,b$, and $[a,b]$ are
independent. Let $\omega $ be the square factor function for $\mathfrak{g}$,
and let $z$ be a nonzero member of $\mathbb{F}$-$\mathrm{sp}(a,b)$ for which
$\omega (z)=0$. Then let $w$ be such that $\mathbb{F}$-$\mathrm{sp}(a,b)=%
\mathbb{F}$-$\mathrm{sp}(z,w)$. The triple $z,w,[z,w]$ is also independent.
Choose $\zeta $ so that not only is $\zeta (z)=0$, but also $\zeta (w)=0$
and $\zeta ([z,w])=1$. We have%
\begin{equation*}
\lbrack z,v]+[v,z]=\omega (v)z+\omega (z)v=\omega (v)z,
\end{equation*}%
by (\ref{comm}). Then%
\begin{equation*}
\zeta ([z,v])+\zeta ([v,z])=\omega (v)\zeta (z)=0,
\end{equation*}%
so that $\zeta ([v,z])=-\zeta ([z,v])$. Define $\zeta ^{\prime }$ by $\zeta
^{\prime }(v)=\zeta ([z,v]$. Then $\zeta $ and $\zeta ^{\prime }$ are
independent, since both are nonzero and $\zeta ^{\prime }(w)=\zeta ([z,w])=1$
but $\zeta (w)=0$. With these arrangements,%
\begin{equation}
\lbrack u,v]_{5}=(\zeta (u)\zeta ^{\prime }(v)-\zeta (v)\zeta ^{\prime
}(u))z.  \label{bracket5simp}
\end{equation}%
The expression $\varphi (u,v)=\zeta (u)\zeta ^{\prime }(v)-\zeta (v)\zeta
^{\prime }(u)$ is a symplectic form of rank 2, and $z$ is in its radical.
(See, for example,~\cite{Grove2001} for background on bilinear forms.) Set
up a basis $u_{1}$, \ldots , $u_{n}$ of $V$ with $\varphi (u_{1},u_{2})=1$,
the radical of $\varphi $ spanned by $u_{3},\ldots ,u_{n}$, and $u_{3}=z$.
Now let $\boldsymbol{\mu }_{5}$ be the structure vector of $\mathfrak{g}_{5}$
relative to the basis $u_{1}$, \ldots , $u_{n}$. Then the nonzero components
$\mu _{ijk}$ of $\boldsymbol{\mu }_{5}$ must have $k=3$. Since $\varphi $ is
symplectic, these nonzero constants are just $\mu _{123}=1$ and $\mu
_{213}=-1$. But this means that $\boldsymbol{\mu }_{5}=\boldsymbol{\eta }$.
Since $\boldsymbol{\eta }$ is in the $G$-orbit of $\Theta (\mathfrak{g}_{5})$%
, we get that $\boldsymbol{\eta }\in \boldsymbol{\lambda }\mathbb{F}G$. As $%
\boldsymbol{\eta }\mathbb{F}G=\mathcal{U}$ (see Remark~\ref{RemUeta}), we
have:

\begin{proposition}
\label{PropLinDegEta} Let $|\mathbb{F}|>2$. Suppose further that $%
\boldsymbol{\lambda }\in \mathcal{M}^{\ast \ast }$ but $\boldsymbol{\lambda }%
\notin \mathcal{M}^{\ast }$. Then $\boldsymbol{\lambda }\looparrowright%
\boldsymbol{\eta}$, so $\mathcal{U}\subseteq \boldsymbol{\lambda }\mathbb{F}%
G $. Moreover, $\mathcal{U}/\mathcal{U}\cap \mathcal{M}^{\ast }$ is
irreducible.
\end{proposition}

We remark in passing that in the special case $\boldsymbol{\lambda }\in%
\mathcal{K}-\mathcal{M}^{\ast}$, the above argument can be simplified. For
such $\boldsymbol{\lambda }$, $[z,z]=0$ and $[u,z]=-[z,u]$. Moreover, there
is a pair $z,w $ with $[z,w]\notin \mathbb{F}$-$\mathrm{sp}(z,w)$. Defining $%
\zeta^{\prime }$ by $\zeta^{\prime }(v)=\zeta ([z,v])$, we see that~%
\eqref{bracket5simp} immediately follows from~\eqref{bracket5}.

\subsection{$\boldsymbol{\protect\lambda}\in \mathcal{C}$}

Let $\boldsymbol{\lambda }\in \mathcal{C}$. Now $[u,z]=[z,u]$. Again we put $%
\zeta ^{\prime }(v)=\zeta ([z,v])$, so that \eqref{bracket5} becomes%
\begin{eqnarray*}
\lbrack u,v]_{5} &=&(\zeta (u)\zeta ^{\prime }(v)+\zeta (v)\zeta ^{\prime
}(u))z \\
&&-\zeta (u)\zeta (v)[z,z]-(\alpha +1)\zeta (u)\zeta (v)\zeta ([z,z])z.
\end{eqnarray*}%
If $\left\vert \mathbb{F}\right\vert >3$, we take $\alpha ^{\prime }\neq
0,1,\alpha $, set up $\mathfrak{g}_{5}^{\prime }$ using $\alpha ^{\prime }$,
take $\mathfrak{g}_{5}^{\prime }-\mathfrak{g}_{5}$, divide by $\alpha
-\alpha ^{\prime }$, and end up with $\mathfrak{g}_{6}$ for which%
\begin{equation*}
\lbrack u,v]_{6}=\zeta (u)\zeta (v)\zeta ([z,z])z.
\end{equation*}%
Assume that $\boldsymbol{\lambda }\notin \mathcal{M}^{\ast \ast }$, so that
for some $z$, $[z,z]=w$ and $z$ are independent. Then we may set $\zeta
(w)=1 $ and have simply%
\begin{equation*}
\lbrack u,v]_{6}=\zeta (u)\zeta (v)z.
\end{equation*}%
Moreover, $V=\mathbb{F}$-$\mathrm{sp}(w)+\ker \zeta $, a direct sum; $z\in
\ker \zeta $. If $v\in \ker \zeta $, then $[u,v]_{6}=0$ for all $u\in V$.
Setting up a basis $u_{1}$, \ldots , $u_{n}$ of $V$ with $u_{1}=w=[z,z]$, $%
u_{2}=z$, and $\ker \zeta $ spanned by $u_{2}$, $u_{3}$, \ldots , $u_{n}$,
we see that $\boldsymbol{\delta }\,(=\mathbf{112})$ belongs to the $G$-orbit
of $\Theta (\mathfrak{g}_{6})$.
Since $\mathcal{N}=\boldsymbol{\delta }(\mathbb{F}G)$ by Proposition~\ref%
{PropNdelta}, we see that $\mathcal{N}\subseteq \boldsymbol{\lambda }(%
\mathbb{F}G)$. In particular, $\mathcal{N}/\mathcal{N}\cap \mathcal{M}^{\ast
\ast }$ is irreducible.

Now suppose that $\mathbb{F}=\mathbb{F}_{3}$. Then the only choice for $%
\alpha $ is $2=-1$, and%
\begin{equation*}
\lbrack u,v]_{5}=(\zeta (u)\zeta ^{\prime }(v)+\zeta (v)\zeta ^{\prime
}(u))z-\zeta (u)\zeta (v)w.
\end{equation*}%
Let $\varphi $ be the bilinear form given by $\varphi (u,v)=\zeta (u)\zeta
^{\prime }(v)+\zeta (v)\zeta ^{\prime }(u)$. Then since $\zeta ^{\prime
}(z)=\zeta ([z,z])=\zeta (w)=1$, we get $\varphi (z,z)=0,\varphi (z,w)=1$,
and $\varphi (w,w)=-\zeta ^{\prime }(w)$. Thus on $\mathbb{F}$-$\mathrm{sp}%
(z,w)$, $\varphi $ is nonsingular. Moreover, the radical of $\varphi $ is $%
R=\ker \zeta \cap \ker \zeta ^{\prime }$. For $z$ and $w$, we have%
\begin{eqnarray*}
\lbrack z,z]_{5} &=&0, \\
\lbrack z,w]_{5} &=&\varphi (z,w)z-\zeta (z)\zeta (w)w=z \\
\lbrack w,w]_{5} &=&-\zeta ^{\prime }(w)z-w.
\end{eqnarray*}%
With $u_{1}=z$, $u_{2}=w$, and $R$ spanned by $u_{3},\ldots ,u_{n}$, we have
that $\boldsymbol{\lambda }_{5}$ belongs to the $G$-orbit of $\boldsymbol{%
\mu }_{5}$, where
\begin{equation*}
\boldsymbol{\mu }_{5}=\mathbf{121}+\mathbf{211}-{\zeta ^{\prime }}(w)\mathbf{%
221}-\mathbf{222}.
\end{equation*}%
If $\zeta ^{\prime }(w)\neq 0$, we can use the transformation $u_{1}\mapsto
-u_{1}$, $u_{i}\mapsto u_{i}$, for $i>1$, to change $\boldsymbol{\mu }_{5}$
to%
\begin{equation*}
\boldsymbol{\mu }_{5}^{\prime }=\mathbf{121}+\mathbf{211+\zeta ^{\prime }}(w)%
\mathbf{221}-\mathbf{222}.
\end{equation*}%
Then $\boldsymbol{\mu }_{5}^{\prime }-\boldsymbol{\mu }_{5}$ scales to $%
\mathbf{221}$. A permutation gives $\mathbf{112}=\boldsymbol{\delta }$
again, and once more $\mathcal{N}\subseteq \boldsymbol{\lambda }(\mathbb{F}%
_{3}G)$.

Finally, suppose that $\zeta ^{\prime }(w)=0$, so that%
\begin{equation*}
\boldsymbol{\mu }_{5}=\mathbf{121}+\mathbf{211}-\mathbf{222}.
\end{equation*}%
Take $g\in G$ with
\begin{equation*}
\lbrack g]=%
\begin{bmatrix}
1 & 0 & 0 & 0 \\
0 & 1 & 0 & 0 \\
0 & 1 & 1 & 0 \\
0 & 0 & 0 & I_{n-3}%
\end{bmatrix}%
,\quad \lbrack g^{-1}]=%
\begin{bmatrix}
1 & 0 & 0 & 0 \\
0 & 1 & 0 & 0 \\
0 & -1 & 1 & 0 \\
0 & 0 & 0 & I_{n-3}%
\end{bmatrix}%
.
\end{equation*}%
Then $\boldsymbol{\mu }_{5}g=\mathbf{121}+\mathbf{211}-\mathbf{222}+\mathbf{%
223}$. Thus $\boldsymbol{\mu }_{5}g-\boldsymbol{\mu }_{5}=\mathbf{223}$, and
a permutation again gets us to $\boldsymbol{\delta }$ and the conclusion
that $\mathcal{N}\subseteq \boldsymbol{\lambda }(\mathbb{F}_{3}G)$. \medskip

Summing up,

\begin{proposition}
\label{PropLinDegDelta} Let $|\mathbb{F}|>2$. Suppose further that $%
\boldsymbol{\lambda }\in \mathcal{C}$ but $\boldsymbol{\lambda }\notin
\mathcal{M}^{\ast\ast }$. Then $\boldsymbol{\lambda }\looparrowright%
\boldsymbol{\delta}$, so $\mathcal{N}\subseteq \boldsymbol{\lambda }\mathbb{F%
}G$. Moreover, $\mathcal{N}/\mathcal{N}\cap\mathcal{M}^{\ast\ast} $ is
irreducible.
\end{proposition}

\section{$GL(V)$-structure of $\boldsymbol{\Lambda }$}

\label{SectGLVstructure}

We assume that $|\mathbb{F}|>2$ throughout this section. Recall that $n$ is
a positive integer with $n\geq 3$. Below, we will use the convention that ``$%
G$-submodule" means ``non-zero proper $G$-submodule".

\subsection{The composition series of $\mathcal{K}$ and $\mathcal{C}$}

\label{SubsecCompSeriesKC}

In~\cite[Section~4.1]{IvanovaPallikaros2019}, under the assumption that $%
\mathbb{F}$ is infinite, all composition series of $\mathcal{K}$ were
obtained in the case $\mathop{\rm char}\nolimits\mathbb{F}\nmid n-1$ and, in
addition, it was shown that in the case \mbox{$\Char\F\mid n-1$} every
composition series for $\mathcal{K}$ begins with $0\subset\mathcal{M}%
^\ast_{(1,-1)}\subset\mathcal{U}$. The techniques used in~\cite%
{IvanovaPallikaros2019} involve the notion of degeneration. In this
subsection we extend these results using linear degeneration and tools like
the adjoint trace form, thus obtaining all the composition series of $%
\mathcal{K}$ for $|\mathbb{F}|>2$. Moreover, we obtain analogous results for
the submodule $\mathcal{C}$. In view of the discussion in Section~\ref%
{SecSubmodulesCK}, this would then provide sufficient information for
determining all the composition factors (with their multiplicities)
occurring in a composition series for~$\boldsymbol{\Lambda}$.

\medskip

We begin by determining all $G$-submodules of $\mathcal{K}$. Let $\mathcal{S}
$ be a $G$-submodule of $\mathcal{K}$ which is not contained in $\mathcal{K}%
\cap \mathcal{M}^{\ast }$. Recall that $\mathcal{K}\cap \mathcal{M}^{\ast }=%
\mathcal{M}_{(1,-1)}^{\ast }$ by Proposition~\ref{PropIntersecM*}. Then, for
any $\boldsymbol{\lambda }\in \mathcal{S}-\mathcal{M}^{\ast }$ we have, by
Proposition~\ref{PropLinDegEta}, that $\boldsymbol{\lambda }\looparrowright
\boldsymbol{\eta }$. Hence $\boldsymbol{\eta }(\mathbb{F}G)\subseteq
\boldsymbol{\lambda }(\mathbb{F}G)\subseteq \mathcal{S}$. Now $\boldsymbol{%
\eta }(\mathbb{F}G)=\mathcal{U}$, and $\mathcal{U}$ is a maximal $G$%
-submodule of $\mathcal{K}$ since $\mathcal{K}/\mathcal{U}$, which is $G$%
-isomorphic to $\widehat{V}$ by Proposition~\ref{PropIsomDualV} (using the
trace form), is irreducible as a $G$-module. It follows that $\mathcal{S}=%
\mathcal{U}$. Hence, $\mathcal{U}$ is the only $G$-submodule of $\mathcal{K}$
which is not contained in $\mathcal{M}^{\ast }$, and since $\mathcal{K}\cap
\mathcal{M}^{\ast }\,(=\mathcal{M}_{(1,-1)}^{\ast })$ is irreducible as a $G$%
-module (see Section~\ref{SecM*}), we conclude that $\mathcal{U}$ and $%
\mathcal{M}_{(1,-1)}^{\ast }$ are the only $G$-submodules of $\mathcal{K}$.
Invoking Proposition~\ref{PropIntersecM*} we get that $\mathcal{U}\cap
\mathcal{M}_{(1,-1)}^{\ast }=0$ (resp., $\mathcal{M}_{(1,-1)}^{\ast }\subset
\mathcal{U}$) if $\mathop{\rm char}\nolimits\mathbb{F}\nmid n-1$ (resp., $%
\mathop{\rm char}\nolimits\mathbb{F}\mid n-1$).
So,

\begin{itemize}
\item If $\mathrm{char}\mathbb{F}\nmid n-1$, then $\mathcal{K}=\mathcal{U}%
\oplus \mathcal{M}_{(1,-1)}^{\ast }$ as a direct sum of irreducible $G$%
-modules (in particular $\mathcal{K}$ has precisely two composition series).

\item If $\mathrm{char}\mathbb{F}\mid n-1$, then we have the unique
composition series $0\subset \mathcal{M}_{(1,-1)}^{\ast }\subset \mathcal{U}%
\subset \mathcal{K}$.
\end{itemize}

Note that the above results are also in line with Proposition~\ref%
{PropLinDegEta} that $\mathcal{U}/\mathcal{U}\cap\mathcal{M}^\ast$ is an
irreducible $G$-module.

\medskip

Our next aim is to determine all $G$-submodules of $\mathcal{C}$. For this,
we let $\mathcal{S}$ be a $G$-submodule of $\mathcal{C}$ which is not
contained in $\mathcal{C}\cap \mathcal{M}^{\ast \ast }$. Then, for any $%
\boldsymbol{\lambda }\in \mathcal{S}-\mathcal{M}^{\ast \ast }$ we have, by
Proposition~\ref{PropLinDegDelta}, that $\boldsymbol{\lambda }%
\looparrowright \boldsymbol{\delta }$. Now $\mathcal{N}=\boldsymbol{\delta }%
\mathbb{F}G$ by Proposition~\ref{PropNdelta}, so $\mathcal{N}\subseteq
\boldsymbol{\lambda }(\mathbb{F}G)\subseteq \mathcal{S}$. But $\mathcal{N}$
is a maximal $G$-submodule of $\mathcal{C}$ since $\mathcal{C}/\mathcal{N}$
is irreducible as a $G$-module (see Proposition~\ref{PropIsomDualV}). Hence $%
\mathcal{S}=\mathcal{N}$. We conclude that the only $G$-submodule of $%
\mathcal{C}$ which is not contained in $\mathcal{C}\cap \mathcal{M}^{\ast
\ast }$ is $\mathcal{N}$.

\medskip

We consider the case $\mathop{\rm char}\nolimits\mathbb{F}\ne2$ first. Then,
by Remark~\ref{RemCM**}(ii), $%
\mathcal{C}\cap\mathcal{M}^{\ast\ast}=\mathcal{C}\cap\mathcal{M}^{\ast}=\mathcal{M}^{\ast}_{(1,1)}$.
Recalling
that $\mathcal{M}^{\ast}_{(1,1)}$ is irreducible, we get that $\mathcal{N}$
and $\mathcal{M}^{\ast}_{(1,1)}$ are the only $G$-submodules of $\mathcal{C}$%
. Finally, invoking Proposition~\ref{PropIntersecM*}, we get

\begin{itemize}
\item If $\mathrm{char}\mathbb{F}\nmid n+1$, then $\mathcal{N}\cap \mathcal{M%
}^{\ast \ast }=0$ and $\mathcal{C}=\mathcal{N}\oplus \mathcal{M}%
_{(1,1)}^{\ast }$, again a direct sum of irreducible $G$-modules (in
particular $\mathcal{C}$ has precisely two composition series).

\item If $\mathrm{char}\mathbb{F}\mid n+1$, then we have the unique
composition series $0\subset \mathcal{M}_{(1,1)}^{\ast }\subset \mathcal{N}%
\subset \mathcal{C}$.
\end{itemize}

Observe that the above results agree with Proposition~\ref{PropLinDegDelta}
that $\mathcal{N}/\mathcal{N}\cap\mathcal{M}^{\ast\ast}$ is irreducible.

\medskip

Suppose now that $\mathop{\rm char}\nolimits\mathbb{F}=2$. Then $\mathcal{K}%
\subset\mathcal{C}$ and $\mathcal{C}\cap\mathcal{M}^{\ast\ast}=\mathcal{K}$
(see Remark~\ref{RemCM**}(i)). So the situation now is that $\mathcal{N}$ is
the only $G$-submodule of $\mathcal{C}$ which is not contained in $\mathcal{K%
}$ (and we have already determined all $G$-submodules of $\mathcal{K}$
whenever $|\mathbb{F}|>2$). We conclude that, in characteristic 2, the $G$%
-submodules $\mathcal{N}$, $\mathcal{M}^\ast_{(1,1)}\,(=\mathcal{M}%
^\ast_{(1,-1)})$, $\mathcal{U}$ and $\mathcal{K}$ form a complete list of $G$%
-submodules for $\mathcal{C}$. Moreover, we have the ``diamond'' %
\begin{equation*}
\begin{tabular}{lllll}
&  & $\mathcal{C}$ &  &  \\
& $\diagup $ &  & $\diagdown $ &  \\
$\mathcal{N}$ &  &  &  & $\mathcal{K}$ \\
& $\diagdown $ &  & $\diagup $ &  \\
&  & $\mathcal{U}$ &  &
\end{tabular}%
\end{equation*}%
with $\mathcal{C}/\mathcal{N}$ and $\mathcal{K}/\mathcal{U}$ both $G$%
-isomorphic to $\widehat{V}$. In the other branch, $\mathcal{C}/\mathcal{K}$
is isomorphic to the $G$-module $\Gamma V$, which we considered in Section~%
\ref{SectChar2}.
Since $|\mathbb{F}|>{2}$, $\Gamma V$ is irreducible by Proposition~\ref%
{PropGammaV}.

\begin{remark}
\label{Rem10.1} Suppose that $\mathop{\rm char}\nolimits\mathbb{F}=2$.

(i) Since $\mathcal{M}^\ast_{(1,1)}$ and $\mathcal{U}$ are both contained in
$\mathcal{K}$, the above discussion ensures that $\mathcal{K}$ is the only
other maximal $G$-submodule of $\mathcal{C}$ apart from $\mathcal{N}$. This
provides an alternative (indirect) way of establishing that $\mathcal{C}/%
\mathcal{K}$ (and hence $\mathcal{N}/\mathcal{U}$ also) is an irreducible $G$%
-module (compare Proposition~\ref{PropGammaV}).

(ii) The following can also be deduced from the discussion preceding this
remark: If $n$ is odd, then $\mathcal{C}$ has precisely two$\mathcal{\ }$%
composition series, namely $0\subset \mathcal{M}_{(1,1)}^{\ast }\subset
\mathcal{U}\subset \mathcal{K}\subset \mathcal{C}$ and $0\subset \mathcal{M}%
_{(1,1)}^{\ast }\subset \mathcal{U}\subset \mathcal{N}\subset \mathcal{C}$.
If $n$ is even, then $\mathcal{C}$ has precisely three$\mathcal{\ }$%
composition series, two of them obtained by refining the first factor of the
filtration $0\subset \mathcal{K}\subset \mathcal{C}$ (which is a direct sum
of two irreducible $G$-modules, as we have seen), the third one being $%
0\subset \mathcal{U}\subset \mathcal{N}\subset \mathcal{C}$.
\end{remark}

For the rest of the paper we will concentrate on the filtration $0\subset%
\mathcal{M}^\ast\subset\mathcal{M}^{\ast\ast}\subset\boldsymbol{\Lambda}$
and discuss possible ways of refining this filtration to a composition
series for $\boldsymbol{\Lambda}$, making use of the various $G$-submodules
of $\boldsymbol{\Lambda}$ we have encountered so far. As regards the
degeneration picture, this is a very natural filtration for $\boldsymbol{%
\Lambda}$ to consider: Recall~\cite[Lemmas~4.4 and~5.4]%
{IvanovaPallikaros2019} that, in the case of an infinite field $\mathbb{F}$,
any structure vector in $\mathcal{M}^{\ast\ast}-\mathcal{M}^\ast$
degenerates to $\boldsymbol{\eta}$ and any structure vector in $\boldsymbol{%
\Lambda }-\mathcal{M}^{\ast\ast}$ degenerates to $\boldsymbol{\delta}$.
Moreover, in the present paper, in Proposition~\ref{PropLinDegEta} we have
established a `linear degeneration analogue' of~\cite[Lemma~4.4]%
{IvanovaPallikaros2019} for $|\mathbb{F}|>2$ using transvections, and in
Example~\ref{ExampleApplThLD}(ii), as an immediate application of Theorem~%
\ref{TheorLinDeg}, we obtained a `linear degeneration analogue' of~\cite[%
Lemma~5.4]{IvanovaPallikaros2019} for $|\mathbb{F}|>4$. It will turn out
from the following discussion that, under our standing assumption for this
section that $|\mathbb{F}|>2$, the $G$-modules $\mathcal{M}^{\ast\ast}/%
\mathcal{M}^\ast$ and $\boldsymbol{\eta}(\mathbb{F} G)\,(=\mathcal{U})$
have, up to $G$-isomorphism, the same composition factors. Similarly, for
the $G$-modules $\boldsymbol{\Lambda}/\mathcal{M}^{\ast\ast}$ and $%
\boldsymbol{\delta}(\mathbb{F} G)\,(=\mathcal{N})$.

\subsection{$G$-submodules of $\mathcal{M}^{\ast \ast }$}

Recall that the $G$-module structure of the submodule $\mathcal{M}^{\ast}$
of $\mathcal{M}^{\ast \ast }$ was completely determined in Section~\ref%
{SecM*}. In particular, the modules $\mathcal{M}^{\ast}_P$ (which are
irreducible and $G$-isomorphic to $\widehat V$) constitute a complete list
of $G$-submodules of $\mathcal{M}^\ast$. Moreover, $\mathcal{M}^\ast$ is a
completely reducible $G$-module isomorphic to $\widehat V\oplus\widehat V$.

Arguing as before, and using our results on transvection degenerations, we
can deduce that any $G$-submodule of $\mathcal{M}^{\ast \ast }$ which is not
contained in $\mathcal{M}^{\ast }$ necessarily contains $\mathcal{U}$. One
such submodule is $\mathcal{K}$. Considering the filtration $0\subset
\mathcal{U}\subset \mathcal{K}\subset \mathcal{M}^{\ast \ast }$ we see that $%
\mathcal{M}^{\ast \ast }/\mathcal{U}$ has exactly two composition factors,
both $G$-isomorphic to $\widehat{V}$ (see Proposition~\ref{PropIsomDualV}
and Corollary~\ref{CorolM**K}).

If $\mathrm{char}\mathbb{F}\nmid n-1$ (including $\mathrm{char}\mathbb{F}=0$%
), then $\mathcal{U}\cap \mathcal{M}^{\ast }=0$, by Proposition \ref%
{PropIntersecM*}. The $G$-submodule diagram is (with dimensions to the left
and right)%
\begin{equation*}
\begin{tabular}{lllclll}
&  &  & $\mathcal{M}^{\ast \ast }$ &  &  & $n^{3}/2-n^{2}/2+n$ \\
&  & $\mathbf{\diagup }$ &  & $\mathbf{\diagdown }$ &  &  \\
$n^{3}/2-n^{2}/2-n$ & $\mathcal{U}$ &  &  &  & $\mathcal{M}^{\ast }$ & $2n$
\\
&  & $\mathbf{\diagdown }$ &  & $\mathbf{\diagup }$ &  &  \\
&  &  & $0$ &  &  & $0$%
\end{tabular}%
.
\end{equation*}%
Here $\mathcal{M}^{\ast }$ is isomorphic to $\widehat{V}\oplus \widehat{V}$,
as we described, so of course $\mathcal{M}^{\ast \ast }/\mathcal{U}%
\backsimeq \widehat{V}\oplus \widehat{V}$, too. Moreover, $\mathcal{M}^{\ast
\ast }/\mathcal{M}^{\ast }$ is $G$-isomorphic to $\mathcal{U}$, and $%
\mathcal{U}$ is irreducible under the assumption on~$\mathbb{F}$.

\medskip

If $\mathop{\rm char}\nolimits\mathbb{F}\mid n-1$, then $\mathcal{U}\cap%
\mathcal{M}^{\ast}=\mathcal{M}^{\ast}_{(1,-1)}$, again by Proposition~\ref%
{PropIntersecM*}. It follows that $\mathcal{U}+\mathcal{M}^{\ast}$ is a $G$%
-submodule of $\mathcal{M}^{\ast\ast}$ of codimension $n$. We now have

\begin{equation*}
\begin{tabular}{ccccccc}
&  &  & $\mathcal{M}^{\ast \ast }$ &  &  & $n^{3}/2-n^{2}/2+n$ \\
&  &  & $\mathbf{\mid }$ &  &  &  \\
&  &  & $\mathcal{U+M}^{\ast }$ &  &  & $n^{3}/2-n^{2}/2$ \\
&  & $\mathbf{\diagup }$ &  & $\mathbf{\diagdown }$ &  &  \\
$n^{3}/2-n^{2}/2-n$ & $\mathcal{U}$ &  &  &  & $\mathcal{M}^{\ast }$ & $2n$
\\
&  & $\mathbf{\diagdown }$ &  & $\mathbf{\diagup }$ &  &  \\
&  &  & $\mathcal{M}_{(1,-1)}^{\ast }$ &  &  & $n$ \\
&  &  & $\mathbf{\mid }$ &  &  &  \\
&  &  & $0$ &  &  & $0$%
\end{tabular}%
.
\end{equation*}
Note that we still have that $\mathcal{M}^{\ast\ast}/\mathcal{U}%
\simeq\widehat V\oplus\widehat V$ since $\mathcal{K}/\mathcal{U}$ and $(%
\mathcal{U}+\mathcal{M}^\ast)/\mathcal{U}$ are two distinct $G$-submodules
of $\mathcal{M}^{\ast\ast}/\mathcal{U}$ both of dimension $n$. (Recall that $%
\mathcal{M}^{\ast\ast}/\mathcal{U}$ has exactly two composition factors
which are both $G$-isomorphic to $\widehat V$, so $\mathcal{M}^{\ast\ast}/%
\mathcal{U}$ has to be the direct sum of $\mathcal{K}/\mathcal{U}$ and $(%
\mathcal{U}+\mathcal{M}^\ast)/\mathcal{U}$, with each of these submodules
being isomorphic to $\widehat V$.) The factor module $(\mathcal{U}+\mathcal{M%
}^\ast)/\mathcal{M}^\ast$ is irreducible since it is $G$-isomorphic to $%
\mathcal{U}/\mathcal{M}^\ast_{(1,-1)}$. Note that in this case, again $%
\mathcal{M}^{\ast\ast}/\mathcal{M}^\ast$ has the same composition factors as
$\mathcal{U}$ but now it is not $G$-isomorphic to $\mathcal{U}$ as is easily
seen from the fact that $\widehat V$ appears as a top quotient of $\mathcal{M%
}^{\ast\ast}/\mathcal{M}^\ast$ but not of $\mathcal{U}$.

\subsection{The factor $\boldsymbol{\Lambda}/\mathcal{M}^{\ast\ast}$}

The aim of this last subsection is to refine the last part of the filtration
$0\subset \mathcal{M}^\ast \subset \mathcal{M}^{\ast\ast} \subset
\boldsymbol{\Lambda}$. As a consequence, combining with the results in the
previous subsections, this would enable us to obtain refinements of this
filtration which are in fact composition series for~$\boldsymbol{\Lambda}$.

\medskip

We consider the case $\mathop{\rm char}\nolimits\mathbb{F}\ne2$ first.

If $\mathop{\rm char}\nolimits\mathbb{F}\nmid n+1$, then $\mathcal{N}\cap%
\mathcal{M}^{\ast\ast}\,(=\mathcal{N}\cap\mathcal{M}^{\ast})=0$ by Remark~%
\ref{RemCM**}(iii),
so we obtain the $G$-submodule diagram
\begin{equation*}
\begin{tabular}{lcccccl}
&  &  & $\boldsymbol{\Lambda}$ &  &  &  \\
&  & $\mathbf{\diagup }$ &  & $\mathbf{\diagdown }$ &  &  \\
$\dfrac{n^{3}}2+\dfrac{n^{2}}2-n$ & $\mathcal{N}$ &  &  &  & $\mathcal{M}%
^{\ast\ast }$ & $\dfrac{n^{3}}2-\dfrac{n^{2}}2+n$ \\
&  & $\mathbf{\diagdown }$ &  & $\mathbf{\diagup }$ &  &  \\
&  &  & $0$ &  &  &
\end{tabular}%
\end{equation*}
Here, $\boldsymbol{\Lambda}/\mathcal{M}^{\ast\ast}$ is $G$-isomorphic to $%
\mathcal{N}$, and $\mathcal{N}$ is irreducible under the assumption on~$%
\mathbb{F}$.

\medskip

If $\mathop{\rm char}\nolimits\mathbb{F}\mid n+1$, then $\mathcal{N}\cap%
\mathcal{M}^{\ast\ast}=\mathcal{M}^{\ast}_{(1,1)}$, again by Remark~\ref%
{RemCM**}(iii),
so we have the diagram
\begin{equation*}
\begin{tabular}{lcccccc}
&  &  & $\boldsymbol{\Lambda}$ &  &  &  \\
&  &  & $|$ &  &  &  \\
&  &  & $\mathcal{N}+\mathcal{M}^{\ast\ast}$ &  &  & $n^3-n$ \\
&  & $\mathbf{\diagup }$ &  & $\mathbf{\diagdown }$ &  &  \\
$n^3/2+n^2/2-n$ & $\mathcal{N}$ &  &  &  & $\mathcal{M}^{\ast\ast}$ & $%
n^3/2-n^2/2+n$ \\
&  & $\mathbf{\diagdown }$ &  & $\mathbf{\diagup }$ &  &  \\
&  &  & $\mathcal{M}^{\ast}_{(1,1)}$ &  &  & $n$ \\
&  &  & $|$ &  &  &  \\
&  &  & $0$ &  &  &
\end{tabular}%
\end{equation*}
Note that $(\mathcal{N}+\mathcal{M}^{\ast\ast})/\mathcal{M}^{\ast\ast}$
(which is $G$-isomorphic to $\mathcal{N}/\mathcal{N}\cap\mathcal{M}%
^{\ast\ast}=\mathcal{N}/\mathcal{M}^{\ast}_{(1,1)}$) is irreducible by
Proposition~\ref{PropLinDegDelta}.

Moreover, $\boldsymbol{\Lambda}/(\mathcal{N}+\mathcal{M}^{\ast\ast})$ is $G$%
-isomorphic to $\widehat V$. To see this, we consider the map $\psi=\mathrm{%
tr}+\widetilde{\mathrm{tr}}$ from $\boldsymbol{\Lambda}$ to $\widehat V$.
This is a $G$-homomorphism which is easily seen to be surjective: note that $%
\widehat V$ is irreducible and $\psi(\mathbf{111})=2\hat v_1\ne0$ since $%
\mathop{\rm char}\nolimits\mathbb{F}\ne2$. On $\mathcal{M}^{\ast\ast}$, we
have $\psi(\boldsymbol{\lambda})=(n+1)\omega_{\boldsymbol{\lambda}}=0$ since
$\mathop{\rm char}\nolimits\mathbb{F}\mid n+1$ (see the discussion preceding
Proposition~\ref{PropTIntersM**}). By definition $\mathcal{N}\subseteq\ker%
\mathrm{tr}$, so $\mathcal{N}\subseteq\ker\widetilde{\mathrm{tr}}$ also,
since $\mathrm{tr}=\widetilde{\mathrm{tr}}$ on $\mathcal{C}$. We conclude
that $\mathcal{N}$ and $\mathcal{M}^{\ast\ast}$ are both contained in $%
\ker\psi$ and so $\mathcal{N}+\mathcal{M}^{\ast\ast}\subseteq\ker\psi$.
Since $\dim(\mathcal{N}+\mathcal{M}^{\ast\ast})=n^3-n=\dim\ker\psi$ we
conclude that $\ker\psi=\mathcal{N}+\mathcal{M}^{\ast\ast}$ and hence $%
\boldsymbol{\Lambda}/(\mathcal{N}+\mathcal{M}^{\ast\ast})$ is $G$-isomorphic
to $\widehat V$. The fact that $\widehat V$ appears as a top quotient of $%
\boldsymbol{\Lambda}/\mathcal{M}^{\ast\ast}$ but not as a top quotient of $%
\mathcal{N}$ ensures that these two $G$-modules are not isomorphic this time.

\medskip

Suppose now that $\mathop{\rm char}\nolimits\mathbb{F}=2$. From Remark~\ref%
{RemCM**}(iii) we get the following picture:
\begin{equation*}
\begin{tabular}{lcccccl}
&  &  & $\boldsymbol{\Lambda }$ &  &  &  \\
&  &  & $|$ &  &  &  \\
&  &  & $\mathcal{N}+\mathcal{M}^{\ast \ast }$ &  &  & $n^{3}/2+n^{2}/2+n$
\\
&  & $\mathbf{\diagup }$ &  & $\mathbf{\diagdown }$ &  &  \\
$n^{3}/2+n^{2}/2-n$ & $\mathcal{N}$ &  &  &  & $\mathcal{M}^{\ast \ast }$ & $%
n^{3}/2-n^{2}/2+n$ \\
&  & $\mathbf{\diagdown }$ &  & $\mathbf{\diagup }$ &  &  \\
&  &  & $\mathcal{U}$ &  &  & $n^{3}/2-n^{2}/2-n$ \\
&  &  & $|$ &  &  &  \\
&  &  & $0$ &  &  &
\end{tabular}%
\end{equation*}%
First observe that $(\mathcal{N}+\mathcal{M}^{\ast \ast })/\mathcal{M}^{\ast
\ast }\,(\simeq \mathcal{N}/\mathcal{U})$ is irreducible (see Remark~\ref%
{Rem10.1}(i)). Moreover, $(\mathcal{N}+\mathcal{M}^{\ast \ast })/\mathcal{N}%
\,(\simeq \mathcal{M}^{\ast \ast }/\mathcal{U})$ has precisely two
composition factors, both of which are $G$-isomorphic to $\widehat{V}$
(recall that $\mathcal{U}\subset \mathcal{K}\subset \mathcal{M}^{\ast \ast }$%
).

We consider the case $n$ is even first. Recalling from Section~\ref%
{SecSubmodulesCK} that $\boldsymbol{\Lambda }/\mathcal{C}\simeq \mathcal{K}$
as $G$-modules, we get from Remark~\ref{Rem10.1} that $\boldsymbol{\Lambda }/%
\mathcal{N}$ has precisely three composition factors, one of them $G$%
-isomorphic to $\mathcal{U}$ while the remaining two are $G$-isomorphic to $%
\widehat{V}$. We conclude that in this case $\boldsymbol{\Lambda }/(\mathcal{%
N}+\mathcal{M}^{\ast \ast })$ is $G$-isomorphic to $\mathcal{U}$ (and it is
an irreducible $G$-module).

Finally, suppose that $n$ is odd. Again from Remark~\ref{Rem10.1} we get
that in this case $\boldsymbol{\Lambda }/\mathcal{N}$ has precisely four$%
\mathcal{\ }$composition factors, three $\mathcal{\ }$of which are $G$%
-isomorphic to $\widehat{V}$ while the fourth is $G$-isomorphic to $\mathcal{%
U}/\mathcal{M}_{(1,1)}^{\ast }$. It follows that $\boldsymbol{\Lambda }/(%
\mathcal{N}+\mathcal{M}^{\ast \ast })$ has precisely two composition
factors, which are exactly the two$\mathcal{\ }$composition factors of $%
\mathcal{U}$.

Consider now the $G$-submodule $(\mathcal{T}\cap\widetilde{\mathcal{T}})+%
\mathcal{M}^{\ast\ast}$ of $\boldsymbol{\Lambda}$. Clearly $\mathcal{N}+%
\mathcal{M}^{\ast\ast}\subseteq(\mathcal{T}\cap\widetilde{\mathcal{T}})+%
\mathcal{M}^{\ast\ast}$ since $\mathcal{N}\subset\mathcal{T}\cap\widetilde{%
\mathcal{T}}$. Moreover, invoking Propositions~\ref{PropL/TIsomToHatV} and~%
\ref{PropDefCondM**} and Corollary~\ref{CorTTM**}, we get that $\dim((%
\mathcal{T}\cap\widetilde{\mathcal{T}})+\mathcal{M}^{\ast%
\ast})=(n^3-2n)+(n^3/2-n^2/2+n)-(n^3/2-n^2/2)=n^3-n$. We conclude that $(%
\mathcal{T}\cap\widetilde{\mathcal{T}})+\mathcal{M}^{\ast\ast}$ properly
contains $\mathcal{N}+\mathcal{M}^{\ast\ast}$, and in the filtration $%
0\subset \mathcal{N}+\mathcal{M}^{\ast\ast}\subset(\mathcal{T}\cap\widetilde{%
\mathcal{T}})+\mathcal{M}^{\ast\ast} \subset\boldsymbol{\Lambda}$ the last
two factors are irreducible as $G$-modules.

The above discussion also verifies that in all four subcases considered
above the $G$-modules $\boldsymbol{\Lambda}/\mathcal{M}^{\ast\ast}$ and $%
\mathcal{N}\,(=\boldsymbol{\delta}(\mathbb{F} G))$ have the same composition
factors.

\end{document}